\documentclass[12pt,a4paper]{amsart}
\usepackage{amsthm,amsmath,amssymb,latexsym}
\usepackage{array,booktabs,ragged2e}
\usepackage{graphicx}
\usepackage[margin=0.55cm]{caption}

\newtheorem{theorem}{Theorem}[section]
\newtheorem{corollary}[theorem]{Corollary}
\newtheorem{proposition}[theorem]{Proposition}
\newtheorem{lemma}[theorem]{Lemma}

\theoremstyle{definition}

\newtheorem{example}[theorem]{Example}

\DeclareMathOperator{\Tr}{Tr}
\DeclareMathOperator{\sgn}{sgn}
\DeclareMathOperator{\supp}{supp\,}

\newcommand{\Hom}{\mathrm{Hom}}
\newcommand{\End}{\mathrm{End}}
\newcommand{\GL}{\mathrm{GL}}

\newcommand{\F}{\mathbf{F}}
\newcommand{\Z}{\mathbf{Z}}
\newcommand{\N}{\mathbf{N}}

\newcommand{\Q}{\mathbf{Q}}
\newcommand{\C}{\mathrm{C}}

\newcommand{\ttfrac}[2]{\genfrac{}{}{0pt}{}{#1}{#2}}
\renewcommand{\L}{L}

\newcommand{\Ind}{\big\uparrow}

\newcommand{\ind}{\!\uparrow}
\newcommand{\res}{\!\downarrow}

\renewcommand{\;}{\,;\,}

\newcounter{thmlistcnt}
\newenvironment{thmlist}%
	{\setcounter{thmlistcnt}{0}%
	\begin{list}{\emph{(\roman{thmlistcnt})}}{%
		\usecounter{thmlistcnt}%
		\setlength{\topsep}{0pt}%
		\setlength{\leftmargin}{0pt}%
		\setlength{\itemsep}{-1pt}%
		\setlength{\itemindent}{17pt}}%
	}%
	{\end{list}}%

\title[Foulkes modules  and 
decomposition numbers]{Foulkes modules  and 
decomposition numbers  of the  symmetric group}

\author{Eugenio Giannelli}
\author{Mark Wildon}

\begin{document}

\begin{abstract}
The decomposition matrix of a finite group in prime characteristic $p$
records the multiplicities of its $p$-modular irreducible representations
as composition factors of the reductions modulo $p$ of its 
irreducible representations in characteristic zero.
The main theorem of 
this paper gives a combinatorial description of certain columns of the decomposition
matrices of symmetric groups in odd %prime
characteristic. The result applies to certain blocks of arbitrarily high weight.
It is obtained by
studying the $p$-local structure of certain twists of the permutation module given by the action of the symmetric group of even degree $2m$
on the collection of set partitions of a set of size $2m$ into $m$ sets each of size two.
In particular, the vertices of the indecomposable summands
of all such modules are characterized; these summands form a new family of
indecomposable
$p$-permutation modules for the symmetric group.
As a further corollary it is shown that for every natural number $w$
there is a diagonal Cartan number in a block of the symmetric group of weight~$w$ equal to~$w+1$. 
\end{abstract}

\maketitle

\thispagestyle{empty}

\section{Introduction}
A central open problem in the
representation theory of finite groups is to find the decomposition matrices of symmetric groups.
The main result of this paper gives a combinatorial description of certain columns of these matrices
in odd prime characteristic. This result applies
to certain blocks of arbitrarily high weight. Another notable feature is
that it is obtained almost entirely by
using the methods of  local representation theory.

We use the definitions of 
James' lecture notes
 \cite{James} 
in which the rows of the decomposition
matrix of the symmetric group~$S_n$ of degree~$n$ in prime characteristic~$p$
are labelled by the partitions of $n$, and the columns
by the $p$-regular partitions of $n$, that is, partitions of $n$ with at most $p-1$ parts of any given size.
The entry $d_{\mu\nu}$ of the decomposition matrix
records the number of composition factors of the Specht module $S^\mu$,
defined over a field of characteristic~$p$, 
that are isomorphic
to the simple module $D^\nu$, first defined by James in \cite{JamesIrrepsBPLMS} as the unique top composition
factor of $S^\nu$. 

Given an odd number $p$, a $p$-core $\gamma$ and $k \in \N_0$,
let~$w_k(\gamma)$ denote the minimum number of $p$-hooks that when added to~$\gamma$
give a partition with exactly $k$ odd parts. (We recall the definition of hooks, $p$-cores 
and weights in \S 2 below.) 
% Proposition~\ref{prop:algi} below implies that $w_k(\gamma)$ is well-defined.)
Let~$\mathcal{E}_k(\gamma)$ denote the 
set of partitions with exactly $k$ odd parts that can be obtained from~$\gamma$ by adding $w_k(\gamma)$ disjoint
$p$-hooks.
 Our main theorem is as follows.

%EG: The reviewer prefers the condition with inequality but he doesn't explain why. Therefore I would like to leave the condition as it is.

\begin{theorem}\label{thm:main}
Let $p$ be an odd prime.
Let $\gamma$ be a $p$-core and let \hbox{$k \in \N_0$}. 
Let $n = |\gamma| + w_k(\gamma)p$. 
If $k \ge p$  suppose that 
\[ w_{k-p}(\gamma)\not= w_{k}(\gamma) - 1. \]
Then $\mathcal{E}_k(\gamma)$ is equal to the disjoint union of subsets
$\mathcal{X}_1$, \ldots, $\mathcal{X}_c$ such that each $\mathcal{X}_j$ has
a unique maximal partition $\nu_j$ in the dominance order. 
Each $\nu_j$ is $p$-regular
and the column of the decomposition matrix of $S_n$ in characteristic~$p$ labelled
by $\nu_j$ has $1$s in the rows labelled by partitions in~$\mathcal{X}_j$, and
$0$s in all other rows.
\end{theorem}

We leave it as a simple exercise  to show that $w_k(\gamma)$ is well-defined. 
It may clarify the main hypothesis in Theorem~\ref{thm:main} to remark that since
 $w_k(\gamma) \le w_{k-p}(\gamma)+1$,
we have $w_{k-p}(\gamma)\not= w_k(\gamma)-1$ if and only 
if $w_{k-p}(\gamma) > w_k(\gamma)-1$.

%\begin{figure}[t]
%\begin{center}
%\scalebox{0.5}{\includegraphics{Figure1Hooks842.pdf}}
%\caption{\small $w_0\bigl( (3,1,1) \bigr) = 3$. The thick line shows
%a $3$-hook in (8,4,2)$, the thinner lines show the two other $3$-hooks
%removed en route 
%}
%\end{center}
%\medskip
%\end{figure}

In particular Theorem~\ref{thm:main} implies that
 if $\lambda$ is a maximal partition in $\mathcal{E}_k(\gamma)$
under the dominance order, then the only non-zero entries of the column 
of the decomposition matrix labelled by $\lambda$ 
are $1$s in rows labelled by partitions in $\mathcal{E}_k(\gamma)$.
We give some examples of Theorem~\ref{thm:main} in Example~\ref{ex:ex}. 

Much of the existing work on decomposition matrices of symmetric groups
has concentrated on giving complete information
about blocks of small weight. In contrast,
Theorem~\ref{thm:main}
gives partial information about blocks of arbitrary  weight. 
In
Proposition~\ref{prop:arbweight} we show that there are blocks of 
every weight in which Theorem~\ref{thm:main} completely determines
a column of the decomposition matrix.
%%In \S 7
%%we make Theorem~\ref{thm:main} more explicit in the case $k=0$ by
%%proving the following proposition, which
%%gives a simple greedy algorithm that can be used to find $w_0(\gamma)$ and
%%an element of $\mathcal{E}_0(\gamma)$.
%
%\begin{proposition}\label{prop:algi}
%Let $p$ be an odd number and let $\gamma$ be a $p$-core. Find the smallest 
%odd part of $\gamma$ and, starting at this part, wrap $p$-hooks up the Young diagram of~$\gamma$
%until the diagram of a partition is obtained. Repeat
%this step on the new partition until no odd parts remain.
%The final partition is the unique maximal element of $\mathcal{E}_0(\gamma)$.
%\end{proposition}
%
%In particular, Proposition~\ref{prop:algi} implies
%that $w_0(\gamma)$ is well-defined; the reader
%is invited to find a simple direct proof of this fact. It easily follows that~$w_k(\gamma)$ 
%is well-defined for any $k\in \N_0$. For example,
%the partition $(9,7,6)$ shown in Figure~1 overleaf is obtained
%by applying two steps of the algorithm in Proposition~\ref{prop:algi} 
%to the $p$-core $(4,2,1)$. Two further steps then give the partition $(16,10,6) \in
%\mathcal{E}_0(4,2,1)$ and show that $w_0( (4,2,1) ) = 5$. 

We prove Theorem~\ref{thm:main}
by studying certain twists by the sign character
of the permutation module $H^{(2^m)}$ 
given by the action of $S_{2m}$ on
the collection of all set partitions of $\{1,\ldots, 2m\}$ into~$m$ sets each of size two, defined over a field  $F$.
(Equivalently~$H^{(2^m)}$ is the $FS_{2m}$-module induced from the trivial
module for the imprimitive wreath product $S_2 \wr S_m \le S_{2m}$.)
For $m$, $k \in \N_0$, let
\[ H^{(2^m\; k)} = \bigl( H^{(2^m)} \boxtimes \sgn_{S_k} \bigr) \Ind^{S_{2m+k}}_{S_{2m} \times S_k} \]
where $\boxtimes$ denotes the outer tensor product of two modules. Thus
when $k=0$ we have $H^{(2^m \; k)} = H^{(2^m)}$, and when $m=0$ we have $H^{(2^m \; k)}
= \sgn_{S_{k}}$; if $k = m = 0$ then $H^{(2^m \; k)}$ should be regarded
as the trivial module for the trivial group $S_0$.
It is known that the ordinary characters of  
these modules are 
multiplicity-free (see Lemma~\ref{lemma:FoulkesChars}), but
as one might expect, when $F$
has 
prime characteristic, their structure can be quite intricate.
Our main contribution is Theorem~\ref{thm:vertices} below,
which characterizes the vertices
of indecomposable summands of~$H^{(2^m \; k)}$
when $F$ has 
odd characteristic. 
The outline of the proof of Theorem~\ref{thm:main} given at the
end of this introduction shows how the local information given by Theorem~\ref{thm:vertices} is
translated into our result on decomposition matrices. This step, from local to global, is the key
to the argument.

\newcommand{\thmVertices}{Let $m \in \N$ and let $k \in \N_0$. If $U$ 
 is an indecomposable non-projective summand of  
 $H^{(2^m \; k)}$, 
 defined over a field $F$ of odd characteristic~$p$,
 then~$U$ has as a vertex a Sylow $p$-subgroup~$Q$ of $(S_2 \wr S_{tp}) \times 
 S_{(r-2t)p}$
for some $t \in \N_0$ and $r \in \N$ with $tp \le m$, 
$2t \le r$ and $(r-2t)p \le k$.
Moreover  
the Green correspondent of $U$ admits a tensor factorization
$V \boxtimes W$ as a module for \hbox{$F\bigl((N_{S_{rp}}(Q)/Q) \times S_{2m+k-rp}\bigr)$}, where
$V$ and $W$ are projective, and~$W$ is an 
indecomposable summand of $H^{(2^{m-tp} \; k-(r-2t)p)}$.}
\begin{theorem}\label{thm:vertices}
\thmVertices
\end{theorem}

Theorem~\ref{thm:vertices} is a significant result in its own right.
For odd primes $p$,
it gives the first infinite family of indecomposable $p$-permutation modules
for the symmetric group (apart from Scott modules, which always lie in
principal blocks) whose vertices
are not Sylow $p$-subgroups of Young
subgroups of symmetric groups.

An important
motivation for the proof of Theorem~\ref{thm:vertices}
is \cite{ErdmannYoung}, in which Erdmann uses similar methods to determine the $p$-local
structure of Young permutation modules and to establish their decomposition into Young modules.
Also relevant is \cite{PagetFiltration}, in which Paget shows 
that $H^{(2^m)}$ has a Specht filtration for any field $F$. 
Using Theorem~11 of \cite{WildonMFs}, it follows that
$H^{(2^m \; k)}$ has a Specht filtration for every $k \in \N_0$.
The local behaviour of $H^{(2^m)}$ in characteristic~$2$, which as one would expect
is very different to the case of odd characteristic, was analysed in \cite{Collings};
the projective summands of $H^{(2^m \; k)}$ in characteristic~$2$ are identified
in \cite[Corollary 9]{Murray}.
In characteristic zero, the 
module $H^{(2^m)}$ arises in the first non-trivial case of Foulkes' Conjecture (see \cite{Foulkes}).
For this reason we call $H^{(2^m)}$ a \emph{Foulkes module} 
and $H^{(2^m\;k)}$ a \emph{twisted Foulkes module}. For some recent results on the characters
of general Foulkes modules 
we refer the reader to \cite{Giannelli} and~\cite{PagetWildon}.

%EG: Do we really want to remove the entire background on decompoosition numbers? I don't....

\subsection*{Background on decomposition numbers}
The problem of finding decomposition numbers for symmetric groups in prime characteristic
has motivated many deep results relating the representation theory of symmetric groups to other
groups and algebras. 
Given the  depth 
of the subject we  
give only a very brief survey, 
 concentrating
on results that apply to Specht modules in blocks of arbitrarily high~weight.

Fix an infinite field $F$  of prime characteristic $p$.
In \cite{JamesTensors}
James proved that 
the decomposition matrix for~$S_n$ modulo $p$ appears, up to a column reordering, 
as a submatrix of the
decomposition matrix for polynomial representations of 
$\GL_d(F)$ of degree~$n$, for any $d \ge n$. In~\cite[6.6g]{GreenGLn}
Green gave an alternative proof of this using the Schur functor from representations
of the Schur algebra to representations of symmetric groups.
James
later established a similar connection
with representations of the finite groups~$\mathrm{GL}_d(\mathbf{F}_q)$,
and the Hecke algebras $\mathcal{H}_{F,q}(S_n)$,
 in the case when $p$ divides $q-1$ (see \cite{JamesGLdDecMatrices}).
In \cite{ErdmannDecNumbers} Erdmann proved, conversely, that
every decomposition number for $\GL_d(F)$
appears as an explicitly determined decomposition number for
some symmetric group.

In \cite{JamesIrrepsBPLMS} James proved that if $D^\nu$ is a composition factor of $S^\mu$
then $\nu$ dominates $\mu$, 
and that if $\mu$ is $p$-regular then $d_{\mu \mu} = 1$.
This establishes the characteristic `wedge' shape of the 
decomposition matrix of $S_n$ with $1$s on its diagonal, shown in the diagram
in \cite[Corollary~12.3]{James}. In \cite{PeelHooks}
Peel proved 
that the hook Specht modules $(n-r,1^r)$ are irreducible when $p$ does not divide $n$, and
described their composition factors for odd primes~$p$ when~$p$ divides $n$. The $p$-regular partitions
labelling these composition factors can be determined by James' 
method of $p$-regularization~\cite{JamesDecII}, 
which gives for each partition $\mu$ of $n$ a $p$-regular partition $\nu$ such
that~$\nu$ dominates $\mu$ and~$d_{\mu\nu} = 1$. In \cite{JamesDecI} and
\cite{JamesChar2}, James determined the decomposition numbers $d_{\mu\nu}$ 
for $\mu$ of the form $(n-r,r)$
and, when $p=2$, of the form $(n-r-1,r,1)$. 
These results were extended by Williams in \cite{Williams}.
In~\hbox{\cite[5.47]{JamesMathas}} James and Mathas, generalizing a conjecture of Carter, conjectured
a necessary and sufficient condition on a partition~$\mu$ for the Specht module~$S^\mu$,
defined for a Hecke algebra $\mathcal{H}_{K,q}(S_n)$ over a field $K$, to be irreducible.
The necessity of this condition was 
proved by Fayers~\cite{FayersReducibleSpechtModules}
for symmetric groups (the case $q=1$), building on earlier work 
of Lyle \cite{LyleReducibleSpechtModules}; later 
Fayers~\cite{FayersIrreducibleSpechtModules}
proved that the condition was sufficient for symmetric groups, and also for
Hecke algebras whenever~$K$ has characteristic zero.
In \cite[Theorem 1.10]{Kleshchev}, Kleshchev determined the decomposition numbers $d_{\lambda \mu}$
when $\mu$ is a $p$-regular partition whose
Young diagram is obtained from the Young diagram of $\lambda$ by moving a single box.
In \cite{WildonCharValues} the second author
proved that in odd characteristic the rows of any decomposition matrix of a symmetric group
are distinct, and so a Specht module is determined, up to isomorphism, by its multiset of composition factors;
in characteristic $2$ the isomorphism~$(S^\mu)^\star = S^{\mu'}$, 
where~$\mu'$ is the conjugate partition to $\mu$,
accounts for all pairs of equal rows in the decomposition matrix.

In \cite{FayersWeight3Abelian} Fayers proved that the decomposition numbers
in blocks of weight~$3$ of abelian defect are either $0$ or $1$. This paper
includes a 
valuable summary of the many techniques for computing decomposition numbers
and references
to earlier results on blocks of weights $1$ and $2$.
For results on weight~$3$ blocks of non-abelian defect,
and blocks of weight $4$,
the reader is referred to
\cite{FayersTanWeight3} and \cite{FayersWeight4}. For further
general results, including branching rules and row and column removals theorems, 
see \cite[Chapter~6, Section~4]{Mathas}.

\subsection*{Outline}
The main tool used to analyse the  structure of twisted Foulkes modules 
over fields of odd characteristic
is the Brauer correspondence for $p$-permutation modules, as developed by Brou{\'e} 
in~\cite{BrouePPerm}. 
We state the necessary background results on the Brauer correspondence and blocks of symmetric groups in~\S 2. 

In \S 3 we collect the 
general results we need on twisted Foulkes modules. In particular, 
Lemma~\ref{lemma:FoulkesChars} gives their ordinary characters.
The twisted Foulkes modules $H^{(2^m\;k)}$ are $p$-permutation modules, but not permutation modules (except when $k \le 1$), and so
some care is needed when applying the Brauer correspondence. 
Our approach is to use Lemma~\ref{lemma:PPermBasis} to construct explicit
$p$-permutation bases: for more theoretical results on monomial modules for finite groups  the reader is referred to~\cite{BoltjeKulshammer}.
 
The main part of the proof begins in
 \S 4 where we prove Theorem~\ref{thm:vertices}. 
In~\S 5 we prove Theorem~\ref{thm:main}, by filling in the details in the following sketch.
The hypotheses of Theorem~\ref{thm:main}, together with 
Lemma~\ref{lemma:FoulkesChars} on the ordinary character of $H^{(2^m \; k)}$,
imply that $H^{(2^m \; k)}$ has a summand in the 
block of~$S_{2m+k}$ with $p$-core $\gamma$. If this summand is non-projective, then it
follows from Theorem~\ref{thm:vertices}, using Theorem~\ref{thm:BrauerBlocks}
on the Brauer correspondence between blocks of symmetric groups,
that either $H^{(2^m \; k-p)}$ has a summand in the block of $S_{2m+k-p}$ with
$p$-core $\gamma$, or one of $H^{(2^{m-p} \; k)}$ and $H^{(2^m \; k-2p)}$ 
has a summand in the block of $S_{2m+k-2p}$
with~$p$-core $\gamma$. All of these are shown to be ruled out by
the hypotheses of Theorem~\ref{thm:main}.
Hence the summand is projective. A short  argument using Lemma~\ref{lemma:FoulkesChars}, 
Brauer reciprocity and Scott's lifting theorem then gives
Theorem~\ref{thm:main}. We also obtain
the proposition below, which identifies a particular projective
summand of $H^{(2^m \; k)}$
in the block of $S_{2m+k}$ with $p$-core~$\gamma$.

\begin{proposition}\label{prop:projectives}
Let $p$ be an odd prime, 
let $\gamma$ be a $p$-core and let \hbox{$k \in \N_0$}. 
If $k \ge p$ suppose that
$w_{k-p}(\gamma) \not= w_k(\gamma)-1$.
Let $2m+k = |\gamma| + w_k(\gamma)p$.
If $\lambda$ is a maximal
partition in the dominance order on $\mathcal{E}_k(\gamma)$ then $\lambda$ is $p$-regular
and the projective cover 
of the simple module $D^\lambda$
is a direct summand of $H^{(2^m \; k)}$,
where both modules are defined over a field of characteristic $p$.
\end{proposition}

In \S 6 we give some further examples and corollaries 
of Theorem~\ref{thm:main} and Proposition~\ref{prop:projectives}. 
In Lemma~\ref{lemma:arbweight} we show that given any odd prime $p$, any $k \in \N_0$,
and any $w \in \N$, there is a $p$-core $\gamma$ such that $w_k(\gamma) = w$.
We use these $p$-cores to show that
the lower bound $c_{\lambda \lambda} \ge w+1$ on the diagonal Cartan numbers in a 
block of weight~$w$, 
proved independently by Richards \cite[Theorem~2.8]{Richards} and 
Bessenrodt and Uno~\cite[Proposition~4.6(i)]{BessenrodtUno}, is attained
for every odd prime $p$ in $p$-blocks of every weight. 
Since the endomorphism algebra of each $H^{(2^m \; k)}$ is commutative
(in any characteristic), it also follows that for any odd prime $p$ and
any $w \in \N$, there is a projective module for a symmetric group lying
in a $p$-block of weight $w$ whose endomorphism algebra is commutative.

%it follows from Proposition~\ref{prop:projectives}
%that for any odd prime $p$ there are projective modules
%for symmetric groups whose endomorphism algebras are commutative,
%%commutative endomorphism algebras
%lying in $p$-blocks of arbitrarily high weight. 
%We also show, in Proposition~\ref{prop:arbweight},
%that 
%We end  in \S 7 by proving Proposition~\ref{prop:algi} using a combinatorial
%argument on James' abacus. 
%In conjunction with the bound on Cartan numbers just mentioned, 
%this proposition makes
%Theorem~\ref{thm:main} completely explicit in many cases. 

%The appendix gives a summary of the notation for twisted Foulkes
%modules which we hope will be helpful as a quick reference for the reader. 

\section{Preliminaries on the Brauer correspondence}

In this section we summarize the principal results from \cite{BrouePPerm}
on the Brauer correspondence for $p$-permutation modules. We then
recall some key facts on the blocks of symmetric groups.
For background on vertices, sources and blocks we refer the
reader to~\cite{Alperin}. Throughout
this section let $F$ be a field of prime characteristic $p$.

\subsection*{The Brauer morphism}

Let $G$ be a finite group. 
An $FG$-module $V$ is said to be a \emph{$p$-permutation} module
if whenever~$P$ is a $p$-subgroup of $G$,
there exists an $F$-basis~$\mathcal{B}$ of $V$ whose
elements are permuted by $P$. We say that $\mathcal{B}$
is a \emph{$p$-permutation basis} of $V$ with respect to $P$,
and write $V = \langle \mathcal{B} \rangle$. %_F$.
It is easily seen that if $V$ has a $p$-permutation basis with
respect to a Sylow $p$-subgroup $P$ of $G$ then $V$ is a $p$-permutation module.

The following proposition characterizing $p$-permutation
modules is proved in \cite[(0.4)]{BrouePPerm}. As usual, if $V$
and $W$ are $FG$-modules we write
$V \mid W$ to indicate that $V$ is isomorphic to a direct summand of $W$.

\begin{proposition}\label{prop:PPermTrivialSouce}
An indecomposable $FG$-module $V$ is a $p$-permutation module 
if and only if there exists a $p$-subgroup $P \le G$ such that $V \mid 
F\ind_P^G$.
\end{proposition}

Thus
an indecomposable $FG$-module is a $p$-permutation
module if and only if it has trivial source. 
It follows that the restriction or induction of a $p$-permutation module
is still $p$-permutation, as is any summand of a $p$-permutation module.

We now recall the definition of the Brauer correspondence for general
$FG$-modules before specializing to $p$-permutation modules.
Let $V$ be an $FG$-module.
Given a $p$-subgroup $Q\le G$ we let $V^Q$ be the set
$\{v\in V : \text{$vg=v$ for all $g \in Q$}\}$ of $Q$-fixed elements. It is easy to see that $V^Q$ is an 
$FN_G(Q)$-module on which $Q$ acts trivially. 
For $R$ a proper subgroup of $Q$, the \emph{relative trace} map 
$\Tr_R^Q:V^R\rightarrow V^Q$ is the linear map defined by 
\[ \Tr_R^Q(v)=\sum_g vg, \]
where the sum is over a set of right-coset 
representatives for $R$ in $Q$. 
We observe that 
\[ \sum_{R<Q}\Tr_R^Q(V^R) \] 
is an $FN_G(Q)$-module contained in $V^Q$.
The \emph{Brauer correspondent} of $V$ with respect to $Q$
is the $FN_G(Q)$-module $V(Q)$ defined by
\[ V(Q) = V^Q / \sum_{R<Q}\Tr_R^Q(V^R). \] 
%EG
It follows immediately from the definition of the Brauer
correspondence that if $U$ is another $FG$-module then
$(U \oplus V)(Q) = U(Q) \oplus V(Q)$.

The following theorem is proved in \cite[3.2(1)]{BrouePPerm}.

\begin{theorem}\label{thm:BrauerCorr}
Let $V$ be an indecomposable $p$-permutation $FG$-module and 
let $Q$ be a vertex of $V$. Let $R$ be a $p$-subgroup of $G$. 
Then $V(R)\neq 0$ if and only if $R\le Q^g$ for some $g\in G$.
\end{theorem}

If $V$ is an $FG$-module with $p$-permutation basis $\mathcal{B}$ with respect 
to a Sylow $p$-subgroup $P$ of $G$ and $R \le P$, then 
taking for each orbit of $R$ on $\mathcal{B}$ the sum
of the vectors in that orbit, we obtain a basis for $V^R$.
The sums over vectors lying in orbits of size $p$ or more
are relative traces from proper subgroups of~$R$, and so $V(R)$ is equal to the $F$-span of
\[ \mathcal{B}^R = \{ v \in \mathcal{B} : \text{$vg = v$ for all $g\in R$} \}. \]
Thus Theorem~\ref{thm:BrauerCorr} has the following
corollary, which we shall use throughout~\S 4. 

\begin{corollary}\label{cor:Brauer}
Let $V$ be a $p$-permutation $FG$-module
with 
$p$-permu\-tation basis
$\mathcal{B}$ with respect to a Sylow $p$-subgroup $P$ of $G$.
Let $R \le P$. The $FN_G(R)$-module $V(R)$ is equal to %the $F$-span of 
$\langle \mathcal{B}^R \rangle$ %. %_F$
and $V$ has an indecomposable summand with a
vertex containing~$R$ if and only if $\mathcal{B}^R \not= \varnothing$.
\end{corollary}

The following proposition
shows that the Brauer correspondent of a $p$-permutation module is again
a $p$-permutation module.
This remark will be crucial in the proof of Theorem \ref{thm:vertices}.

\begin{proposition}
Let $U$ be a $p$-permutation $FG$-module and let $R$ be a $p$-subgroup of $G$. The Brauer correspondent $U(R)$ of $U$ is a $p$-permutation
$FN_G(R)$-module.
\end{proposition}

\begin{proof}
Let $P'$ be a Sylow $p$-subgroup of $N_G(R)$ and let $P$ be a Sylow $p$-subgroup of $G$ containing $P'$. 
Let $\mathcal{B}$ be 
a $p$-permutation basis of $U$ with respect to $P$. By Corollary \ref{cor:Brauer} 
the $FN_G(R)$-module
$U(R)$ has $\mathcal{B}^R$ as a basis. 
Since $R \le N_G(P')$, it follows 
that $\mathcal{B}^R$ is a $p$-permutation basis of $U(R)$ with respect to $P'$.
Therefore $U(R)$ is a $p$-permutation $FN_G(R)$-module.
%Let $P''$ be another Sylow $p$-subgroup of $N_{G}(R)$ and let $g\in N_G(R)$ such that $P''=P'^g$. 
%Then we have that $$\mathcal{B}'':=\{xg\ |\ x\in\mathcal{B}^R \}$$ is a $p$-permutation basis of $U(R)$ with respect to $P''$.
%This completes the proof.
\end{proof}

Remarkably, the Brauer correspondent of an indecomposable $p$-permut\-ation module
with respect to its vertex 
agrees with its Green correspondent. This is proved in 
\cite[Exercise~27.4]{Thevenaz}. We therefore have
the following theorem (see \cite[3.6]{BrouePPerm}).

\begin{theorem}\label{thm:BrauerBij}
The Brauer correspondence sending $V$ to $V(Q)$ is a
bijection between the set of indecomposable $p$-permutation 
$FG$-modules with vertex $Q$ and the set of indecomposable 
projective $FN_G(Q)/Q$-modules. The $FN_G(Q)$-module $V(Q)$ is the
Green correspondent of $V$.
\end{theorem}

\subsection*{Blocks of the symmetric group}

The blocks of symmetric groups are described
combinatorially
by Naka\-yama's Conjecture, first proved by
Brauer and Robinson
in  two connected papers \cite{RobinsonNakayama} and \cite{BrauerNakayama}.
In order to state this result, we must 
recall some definitions. 

Let $\lambda$ be a partition. 
A \emph{$p$-hook} in $\lambda$ is a connected part of the 
rim of the Young diagram of $\lambda$ consisting of exactly $p$ boxes,
whose removal leaves the diagram of a partition. 
By repeatedly removing $p$-hooks from~$\lambda$
we obtain the \emph{$p$-core} of $\lambda$; the number
of hooks we remove is the \emph{weight} of $\lambda$.
%For example, the partition $(9,7,6)$ shown in Figure~1 has $p$-core $(4,2,1)$
%and weight~$3$. 

\begin{theorem}[Nakayama's Conjecture]\label{thm:Nakayama} %LMS change
Let $p$ be prime.
The $p$-blocks of  $S_n$ are labelled by pairs $(\gamma, w)$, where
$\gamma$ is a $p$-core and $w \in \mathbf{N}_0$ is the associated weight,
such that $\left|\gamma\right| + wp = n$.
Thus the Specht module $S^\lambda$ lies
in the block labelled by $(\gamma, w)$ if and only if $\lambda$
has $p$-core~$\gamma$ and weight $w$.\hfill$\Box$
\end{theorem}

We denote the block of weight $w$ corresponding to the $p$-core $\gamma$
by $B(\gamma, w)$.
The following description of 
the Brauer correspondence for blocks of symmetric groups
is critical to the proof of Proposition~\ref{prop:allproj} below.
% will be critical to our argument in \S 5.

\begin{theorem}\label{thm:BrauerBlocks}
Let $V$ be
an indecomposable $p$-permutation module lying in 
the block $B(\gamma, w)$ of $S_n$. Suppose 
that $R$ is contained in a vertex of $V$ and that
$R$ moves exactly the first~$rp$ elements of 
$\{1,\ldots,n\}$; that is~$\supp(R) = \{1,\ldots,rp\}$.
Then $N_{S_n}(R) \cong N_{S_{rp}}(R) \times S_{n-rp}$. Moreover,

\begin{thmlist}
\item $N_{S_{rp}}(R)$ has a unique block, $b$ say.

\item The blocks $b \otimes B(\gamma,w-r)$ and $B(\gamma,w)$ are Brauer correspondents.

\item As an $FN_{S_n}(R)$-module, $V(R)$ lies 
in $b \otimes B(\gamma,w-r)$. 
\end{thmlist}
\end{theorem}

\begin{proof}
Part (i) is an immediate corollary of Lemma 2.6 and the following
sentence of \cite{BroueBlocks}. Part (ii)
is stated in (2) on page 166 of \cite{BroueBlocks},
and then proved as a corollary of the characterisation of 
maximal Brauer pairs
given in Proposition 2.12 of \cite{BroueBlocks}. Part (iii)
follows from Lemma 7.4 
of \cite{MWvertices}.
\end{proof}
%\smallskip

%Proposition 2.12 in \cite{BroueBlocks} also proves the following
%theorem on defect groups, which we shall need in \S 4 below.
% (See \cite{JK} Theorem~6.2.45 for a different approach.)
%
%\begin{theorem} 
%\label{Thm:SymDefectGroups}
%The defect group of a $p$-block of a symmetric group of weight~$w$ is a Sylow
%$p$-subgroup of $S_{wp}$.
%\end{theorem}

\section{Foulkes modules and twisted Foulkes modules}
Throughout this
section let $F$ be a field and let $m \in \N$, $k \in \N_0$. 
We define~$\Omega^{(2^m)}$ to be the collection of all 
set partitions of $\{1,\ldots,2m\}$ into~$m$ sets each of size two. 
The symmetric group~$S_{2m}$ acts
on $\Omega^{(2^m)}$ in an obvious way. We have already defined the Foulkes module
$H^{(2^m)}$ to be the permutation module 
with $F$-basis $\Omega^{(2^m)}$, and the twisted Foulkes module
$H^{(2^m\; k)}$ to be
$\bigl( H^{(2^m)} \boxtimes \sgn_{S_k} \bigr) \ind^{S_{2m+k}}_{S_{2m} \times S_k}$.

Let $\chi^\lambda$ denote the irreducible character of $S_n$ corresponding
to the partition $\lambda$ of $n$.
When $F$ has characteristic zero the ordinary character of $H^{(2^m)}$ was
found by Thrall \cite[Theorem~III]{Thrall} to be $\sum_\mu \chi^{2\mu}$
where the sum is over all partitions~$\mu$ of $m$ and $2\mu$ is the partition
obtained from~$\mu$ by doubling  each part. 

\begin{lemma}\label{lemma:FoulkesChars}
The ordinary character 
of $H^{(2^m \; k)}$ is $\sum_\lambda \chi^\lambda$, where
the sum is over all partitions $\lambda$ of $2m+k$ with exactly $k$ odd parts.
\end{lemma}

\begin{proof}
This follows from Pieri's rule (see \cite[7.15.9]{StanleyII}) applied to the
ordinary character of~$H^{(2^m)}$.
\end{proof}

We remark that an alternative proof of
Lemma~\ref{lemma:FoulkesChars} with minimal pre-requisites
can be found in \cite{IRS}; the main result of \cite{IRS} uses the
characters of twisted Foulkes modules to construct a `model' character
for each symmetric group containing each irreducible character exactly once.

In the remainder of this section we suppose that $F$
has odd characteristic $p$ and define a module
isomorphic to $H^{(2^m \; k)}$ that will be used in the calculations in \S 4.
Let $S_X$ denote the symmetric group on the set $X$.
Let $\Delta^{(2^m \; k)}$ be 
the set of all elements of the form
\[ 
\bigl\{ \{i_1, i'_1\}, \ldots, \{ i_m, i'_m \} ,
(j_1, \ldots, j_k)\bigr\}
\]
where 
$\{i_1, i'_1, \ldots, i_m, i'_m,  
j_1, \ldots, j_k\} = \{1,\ldots, 2m+k\}$. 
Given $\delta \in \Delta^{(2^m \; k)}$
of the form above, we define 
\begin{align*}
\mathcal{S}(\delta) &= 
\bigl\{ \{i_1, i'_1\}, \ldots, \{ i_m, i'_m \} \bigr\}, \\
\mathcal{T}(\delta) &= \{ j_1, \ldots, j_k\}.
\end{align*}
The symmetric group
$S_{2m+k}$ acts transitively on $\Delta^{(2^m \; k)}$ 
by
\[
\delta g =
\bigl\{ \{i_1g, i'_1g\}, \ldots, \{ i_mg, i'_mg \} ,
(j_1g, \ldots, j_kg)\bigr\}
\] 
for $g \in S_{2m+k}$.
Let $F\Delta^{(2^m \; k)}$ be the 
permutation
module for $FS_{2m+k}$ with $F$-basis $\Delta^{(2^m \; k)}$. Let
$K^{(2^m \; k)}$ be the subspace of $F\Delta^{(2^m \; k)}$  spanned by 
\[ \bigl\{ \delta - \sgn(g) \delta g : \delta \in \Delta^{(2^m \; k)}, 
g\in S_{\mathcal{T}(\delta)} \bigr\}. \]
Since this set is permuted by $S_{2m+k}$, it is clear that
$K^{(2^m \; k)}$ is an $FS_{2m+k}$-submodule of~$F\Delta^{(2^m \; k)}$.
For $\delta \in \Delta^{(2^m \; k)}$, let $\overline{\delta} \in
F\Delta^{(2^m \; k)} / K^{(2^m \; k)}$ denote the image 
$\delta + K^{(2^m \; k)}$ of~$\delta$
under the quotient map. 
Let $\Omega^{(2^m \; k)}$
be the subset of~$\Delta^{(2^m \; k)}$ consisting of those
elements of the form above such that $j_1 < \ldots < j_k$. 
In the next lemma we use  $\Omega^{(2^m \; k)}$
to identify $F\Delta^{(2^m \; k)}/ 
K^{(2^m \; k)}$
with~$H^{(2^m \; k)}$.

\begin{lemma}{\ }\label{lemma:DeltaKIso}

\begin{thmlist}
\item For each $\delta \in \Delta^{(2^m \; k)}$ there
exists a unique $\omega \in \Omega^{(2^m \; k)}$ 
such that $\overline{\delta} \in \{\overline{\omega}, -\overline{\omega}\}$.
Moreover, for this $\omega$ we have 
$\mathcal{S}(\delta) = \mathcal{S}(\omega)$ and
$\mathcal{T}(\delta) = \mathcal{T}(\omega)$ and there
exists a unique
$h \in S_{\mathcal{T}(\delta)}$  such that $\delta h = \omega$.

%\vbox{
\item The set $\{ \overline{\omega} :
\omega \in \Omega^{(2^m \; k)} \}$
is an $F$-basis for $F\Delta^{(2^m \; k)} / K^{(2^m \; k)}$.

\item The $FS_{2m+k}$-modules $H^{(2^m \; k)}$ and $F\Delta^{(2^m \; k)} / K^{(2^m \; k)}$
are isomorphic. 
\end{thmlist}
\end{lemma}

\begin{proof}
For brevity we write $K$ for $K^{(2^m \; k)}$.
Let 
\[ \delta = \bigl\{ \{i_1, i'_1\}, \ldots, \{ i_m, i'_m \},
(j_1, \ldots, j_k)\bigr\} \in \Delta^{(2^m \; k)}. \]
%%EG
%Suppose that exist $\omega_1$ and $\omega_2$ in $\Omega^{(2^m;k)}$ such that 
%$$\overline{\delta}\in \{\overline{\omega_1}, -\overline{\omega_1}\}\cap\{\overline{\omega_2}, -\overline{\omega_2}\}.$$
%This implies that $\overline{\omega_1}=\pm\overline{\omega_2}$, and therefore that exist $g\in S_{\mathcal{T}(\omega_1)}$ such that 
%$\omega_2=\omega_1g$. Hence $\mathcal{T}(\omega_1)=\mathcal{T}(\omega_2)$ and $\mathcal{S}(\omega_1)=\mathcal{S}(\omega_2)$; this is enough to deduce that $\omega_1=\omega_2$ by definition of $\Omega^{(2^m;k)}$. 
%%EGend
The 
unique $h \in S_{\mathcal{T}(\delta)}$ such that $\delta h \in \Omega^{(2^m\; k)}$
is the permutation such that 
$j_1 h < \ldots < j_k h$. Set $\omega = \delta h$. Since $S(\omega) = S(\delta)$ 
and $T(\omega) = T(\delta)$ we have the existence part of (i).
Since $\delta - \sgn(h)\delta h \in K$, 
it follows that $\overline{\delta} = \sgn(h) \overline{\omega}$, and that %\delta h}$, and that
$F\Delta^{(2^m \; k)} / K$ is spanned by \hbox{$\{ \overline{\omega} :
\omega \in \Omega^{(2^m \; k)} \}$}. 
Set~$x = h^{-1}$.
If $g \in S_{\mathcal{T}(\delta)}$  then
\[ \delta - \sgn(g) \delta g = -\sgn(x) \bigl( \omega - \sgn(x) \omega x \bigr)
+ \sgn(x) \bigl( \omega - \sgn(xg) \omega xg \bigr). \]
Since $\mathcal{T}(\delta) = \mathcal{T}(\omega)$ we have $x$, $xg \in S_{\mathcal{T}(\omega)}$.
It follows that
$K$ is spanned by 
\[ \{ \omega - \sgn(y) \omega y : \omega
 \in \Omega^{(2^m \; k)}, y \in S_{\mathcal{T}(\omega)} \}. \] 
Hence $\dim F\Delta^{(2^m \; k)} / K \le | \Omega^{(2^m \; k)}|$
and $\dim K \le | \Omega^{(2^m \; k)}| (k! -1)$. Since 
%\[ 
$\dim F\Delta^{(2^m \; k)}
= |\Omega^{(2^m \; k)}|\hskip1pt k!$ %, 
%\]
we have equality in both cases.
This proves part~(ii). Moreover, if $\overline{\omega}
= \pm \overline{\omega'}$ for $\omega, \omega \in \Omega^{(2^m \; k)}$ then
$S(\omega) = S(\omega')$ and~$T(\omega) = T(\omega')$, and so $\omega
= \omega'$. This proves the uniqueness in (i).
 % of the lemma.

For (iii), let
\[ \omega = \big\{ \{1,2\}, \ldots, \{2m-1,2m\}, (2m+1,\ldots, 2m+k) \bigr\} \in
\Omega^{(2^m \; k)}. \]
Write $S_{2m} \times S_k$ for $S_{\{1,\ldots, 2m\}} 
\times S_{\{2m+1,\ldots, 2m+k\}}$, thought of as a subgroup of $S_{2m+k}$.
Given $h \in S_{\{1,\ldots, 2m\}}$ 
and $x\in S_{\{2m+1, \ldots, 2m+k\}}$ we have $\overline{\omega} hx = 
\sgn(x) \overline{\omega h}$.
By (ii) the set $\{ \overline{\omega h} : h \in S_{2m} \}$
is linearly independent.
Hence the $F(S_{2m} \times S_k)$-submodule of $F\Delta^{(2^m \; k)} / K$ generated 
by~$\overline{\omega}$ is isomorphic to $H^{(2^m)} \boxtimes \sgn_{S_k}$. 
Since \[ \dim F\Delta^{(2^m \; k)} / K = 
|\Omega^{(2^m \; k)}| = \binom{2m+k}{k} \dim H^{(2^m)} \]
and the index of $S_{2m} \times S_k$ in $S_{2m+k}$ is $\binom{2m+k}{k}$,
it follows that
\[ F\Delta^{(2^m \; k)} / K \cong \bigl( H^{(2^m)} \boxtimes \sgn_{S_k} 
\bigr) \Ind^{S_{2m+k}}_{S_{2m} \times S_k} \]
as required.
\end{proof}
%EG
Since $p$ is odd we have that 
\[ F\uparrow_{(S_2\wr S_m)\times A_k}^{(S_2\wr S_m)\times S_k}=F_{(S_2\wr S_m)\times S_k}\oplus \big(F_{S_2\wr S_m}\boxtimes \mathrm{sgn}_{S_k}\big)\]
where~$A_k$ denotes the alternating group on 
$\{2k+1,\ldots, 2k+n\}$. Therefore $H^{(2^m \; k)}$ is a direct summand of
the module induced from the trivial \hbox{$F\bigl( (S_2 \wr S_m) \times A_k \bigr)$-module},
and so $H^{(2^m \; k)}$ is a $p$-permutation module.

In the following lemma we construct a $p$-permutation basis for~$H^{(2^m \; k)}$.

\medskip
\begin{lemma}\label{lemma:PPermBasis}
Let $P$ be a $p$-subgroup of $S_{2m+k}$. 

\begin{thmlist}
\item There 
is a choice of signs $s_\omega \in \{+1, -1\}$ for $\omega \in 
\Omega^{(2^m \; k)}$ such that $\{ s_\omega \overline{\omega} :
\omega \in \Omega^{(2^m \; k)} \}$ is 
a $p$-permutation basis for $H^{(2^m \; k)}$ with respect to~$P$.

\item Let $\omega = \bigl\{ \{i_1, i'_1\}, 
\ldots, \{i_m, i'_m\}, (j_1, \ldots, j_k) \} \in \Omega^{(2^m \; k)}$
and let $g \in P$.
Then $\overline{\omega}$ is fixed by~$g$ if and only 
if $\bigl\{
 \{i_1, i'_1\}, \ldots, \{i_m, i'_m\} \bigr\}$ is fixed by $g$.
\end{thmlist}
\end{lemma}

\begin{proof}
For $\omega \in \Omega^{(2^m \; k)}$ let
$\gamma(\omega) = \bigl( \mathcal{S}(\omega), \mathcal{T}(\omega) \bigr)$.
Let 
\[ \Gamma^{(2^m \; k)} = \bigl\{ \gamma(\omega) : \omega \in \Omega^{(2^m \; k)} \bigr\}. \]
%EG
Notice that the map $$\gamma:\Omega^{(2^m;k)}\longrightarrow \Gamma^{(2^m \; k)}$$
associating to each $\omega\in\Omega^{(2^m;k)}$ the element $\gamma(\omega)\in\Gamma^{(2^m \; k)}$
is a bijection. %EGend
The set $\Gamma^{(2^m \; k)}$ is acted on by $S_{2m+k}$ in the obvious way. 
Let $\gamma_1, \ldots, \gamma_c \in \Gamma^{(2^m \; k)}$ be
representatives for the orbits of $P$ on $\Gamma^{(2^m \; k)}$. 
For each $b \in \{1,\ldots, c\}$, let 
$\omega_b \in \Omega^{(2^m \; k)}$ be the unique element such that
$\gamma_b = \gamma(\omega_b)$. Given any $\omega \in \Omega^{(2^m\;k)}$
there exists a unique $b$ such that $\gamma(\omega)$ is in the orbit of $P$
on $\Gamma^{(2^m \; k)}$
containing~$\gamma_b$. Choose $g \in P$
such that $\gamma(\omega) = \gamma_b g$. Then
$\omega$ and $\omega_b g$ are equal up to the order of the numbers
in their $k$-tuples, and so there exists $h \in S_{\mathcal{T}(\omega)}$ 
such that  $\omega_b g h = \omega$. By Lemma~\ref{lemma:DeltaKIso}(i)
we have
\[ \overline{\omega_b} g = s_\omega \overline{\omega} \]
for some $s_\omega \in \{+1, -1\}$. 
If $\tilde{g} \in P$ is another permutation such that $\gamma(\omega)
= \gamma_b \tilde{g}$ then 
$\overline{\omega_b} g{\tilde{g}}^{-1} = \pm \overline{\omega_b}$.
Hence the $F$-span of $\overline{\omega_b}$ is a $1$-dimensional
representation
of the cyclic $p$-group generated by $g{\tilde{g}^{-1}}$. The unique
such representation is the trivial one, so $\overline{\omega_b}
g = \overline{\omega_b}\tilde{g}$. The sign $s_\omega$ is therefore well-defined.
Now suppose that $\omega$, $\omega' \in \Omega^{(2^m \; k)}$
and $h \in P$ are such that~$s_{\omega}\overline{\omega}h = \pm s_{\omega'}\overline{\omega'}$.
By construction of the basis there exists $\omega_b \in \Omega^{(2^m \; k)}$ and
$g$, $g' \in P$ such that $s_\omega \overline{\omega} = \overline{\omega_b} g$
and $s_{\omega'} \overline{\omega'} = \overline{\omega_b} g'$.
Therefore
\[ \overline{\omega_b} g h = s_\omega \overline{\omega} h = \pm s_{\omega'}
\overline{\omega'} = \pm \overline{\omega_b} g'\]
and so $\overline{\omega_b} gh{g'}^{-1} = \pm \overline{\omega_b}$.
As before, the plus sign must be correct. This proves (i).

%%EG
%Moreover, suppose for a contradiction that exist $\omega,\delta\in\Omega^{(2^m;k)}$ and $g\in P$ such that $\gamma(\omega)g=\gamma(\delta)$ but $$s_\omega\overline{\omega}g=-s_\delta\overline{\delta}.$$
%Then, from the construction used above to define the signs $s_\eta$ for all $\eta\in\Omega^{(2^m;k)}$, we have that there exists an element $\omega_b\in\Omega^{(2^m;k)}$ such that $\overline{\omega_b}h_1=s_\omega\overline{\omega}$ and $\overline{\omega_b}h_2=s_\delta\overline{\delta}$ for some $h_1, h_2\in P$. Therefore $$(\overline{\omega_b}h_1)g=(s_\omega\overline{\omega})g=-s_\delta\overline{\delta}=-\overline{\omega_b}h_2.$$ As already discussed above, this is a contradiction. Hence, $\{ s_\omega \overline{\omega} :
%\omega \in \Omega^{(2^m \; k)} \}$ is permuted by $P$. Part (i) now follows from Lemma \ref{lemma:DeltaKIso}(ii).
%%EGend

For (ii), 
suppose that $\overline{\omega}g = \overline{\omega}$.
Setting $\delta = \omega g$, and noting that $\overline{\delta} =
\overline{\omega}$,
it follows from Lemma~\ref{lemma:DeltaKIso}(i)
that $\mathcal{S}(\omega g) = \mathcal{S}(\delta) = \mathcal{S}(\omega)$.
Hence the condition in (ii)
is necessary.
Conversely, if
$\bigl\{ \{i_1, i'_1\}, \ldots, \{i_m, i'_m\} \bigr\}$
is fixed by~$g$ then~$g$ permutes $\{j_1,\ldots, j_k\}$ 
and so
$\overline{\omega} g \in \{ \overline{\omega}, -\overline{\omega} \}$.
Since $g \in P$, it now follows from~(i) that
$\overline{\omega} g = \overline{\omega}$, as required.
\end{proof}

In applications of Lemma~\ref{lemma:PPermBasis}(ii) 
it will be useful to
note that there is an isomorphism of $S_{2m}$-sets
between $\Omega^{(2^m)}$ and the set of fixed-point free
involutions in $S_{2m}$, where the symmetric group acts
by conjugacy. 
Given $\omega \in \Omega^{(2^m \; k)}$ with
$\mathcal{S}(\omega) = \bigl\{ \{i_1, i'_1 \}, \ldots, \{i_m, i'_m\} \bigr\}$, 
we define 
\[ \mathcal{I}(\omega) = (i_1,i'_1) \cdots (i_m, i'_m) \in S_{2m+k}. \]
By Lemma~\ref{lemma:PPermBasis}(ii), if $g \in S_{2m+k}$ is a $p$-element
then $g$ fixes $\overline{\omega}$ if and only if~$g$ commutes with~$\mathcal{I}(\omega)$.
 Corollary~\ref{cor:Brauer} and Lemma~\ref{lemma:PPermBasis}
 therefore implies the following proposition,
 which we shall use repeatedly in the next section.

\begin{proposition}\label{prop:Brauer}
Let $R$ be a $p$-subgroup of $S_{2m+k}$ and let $P$ be a Sylow $p$-subgroup
of $S_{2m+k}$ containing a Sylow $p$-subgroup of $N_G(R)$.
There is a choice of signs $s_\omega \in \{+1,-1\}$ for $\omega \in \Omega^{(2^m \; k)}$
such that
\[ \bigl\{ s_\omega \overline{\omega} : \omega \in \Omega^{(2^m \; k)}, \,
\mathcal{I}(\omega) \in \C_{S_{2m+k}}(R) \bigr\}. \]
is a $p$-permutation basis for the
Brauer correspondent $H^{(2^m \; k)}(R)$ with respect to $P \cap N_G(R)$.
\end{proposition}

\section{The local structure of $H^{(2^m\;k)}$}

In this section we  prove Theorem~\ref{thm:vertices}. 
Throughout we let $F$ be a field of odd characteristic $p$
and fix $m \in \N$, $k\in \N_0$.
Any vertex of an indecomposable non-projective summand of $H^{(2^m \; k)}$ must
contain, up to conjugacy,
one of the subgroups
\[ R_r = \langle z_1z_2\cdots z_r \rangle \]
where $z_j$ is the $p$-cycle $(p(j-1)+1, %p(j-1)+2,
\ldots,pj)$ and $rp \le 2m+k$,
so we begin by calculating $H^{(2^m \; k)}(R_r)$. 
In the second step
we show that, for any $t \in \N$ such that $2t \le r$,
the Brauer correspondent $H^{(2^{tp} \; (r-2t)p)}(R_r)$ is indecomposable
as an $FN_{S_{rp}}(R_r)$-module 
and determine its vertex; in the third step 
we combine these results to complete the proof.

In the second step we shall require the basic lemma below; its proof
is left to the reader.

\begin{lemma}\label{lemma:perm} If $P$ is a $p$-group and $Q$ is a subgroup of $P$ then
the permutation module $F\ind_Q^P$ is indecomposable with vertex $Q$.
\end{lemma}

\subsection*{First step: Brauer correspondent with respect to $R_r$}

Let $r \in \N$ be such that $rp \le 2m+k$. We define
\[ T_r = \{ t \in \N_0 : tp\leq m,\ 2t \le r,\ (r-2t)p \le k \}. \]
For $t \in T_r$ let
\[ \mathcal{A}_{2t} 
=
\Bigl\{ \overline{\omega} \,\, : \,\, 
\ttfrac{\omega \in \Omega^{(2^m \; k)},\mathcal{I}(\omega) \in \C_{S_{2m+k}}(R_r) \hfill}
{\text{$\supp \mathcal{I}(\omega)$ contains exactly $2t$ orbits of $R_r$ of length $p$ 
}} \Bigr\}. \]

\begin{lemma}\label{lemma:Rcorr}
There is a direct sum decomposition of $FN_{S_{2m+k}}(R_r)$-modules
\[ H^{(2^m \; k)}(R_r) \cong \bigoplus_{t\in T_r} \langle \mathcal{A}_{2t}
\rangle. \] %_F . \]
\end{lemma}

\begin{proof}
By Proposition~\ref{prop:Brauer}
the $FN_{S_{2m+k}}(R_{r})$-module $H^{(2^m;k)}(R_r)$ 
has as a basis %the set
\[ \mathcal{A} = 
\bigl\{ \overline{\omega} : \omega \in \Omega^{(2^m \; k)},\ 
\mathcal{I}(\omega) \in \C_{S_{2m+k}}(R_r) \bigr\}. \]
Let $\omega \in \Omega^{(2^m \; k)}$ be such that $\mathcal{I}(\omega)
\in \C_{S_{2m+k}}(R_r)$. Then
$\mathcal{I}(\omega)$ permutes, as blocks for its action,
the orbits of $R_r$. It follows that the number of orbits of $R_r$ of length $p$
contained in 
 $\supp \mathcal{I}(\omega)$ is even. Suppose this number is~$2t$. Clearly $2t \le r$
 and $2tp \le 2m$.
% contains exactly an even number, $2t$ say, of orbits of $R_r$ of length $p$. Clearly $2t \le r$. 
The remaining $r-2t$ orbits of length $p$ are contained in $\mathcal{T}(\omega)$.
Thus $(r-2t)p \le k$, and so $t \in T_r$ and 
$\overline{\omega} \in \mathcal{A}_{2t}$.

Let the $p$-cycles corresponding to the $2t$ orbits of $R_r$
that are contained in $\supp \mathcal{I}(\omega)$ be
 $z_{j_1}$, \ldots, $z_{j_{2t}}$.
Let $g\in N_{S_{2m+k}}(R_{r})$. Let 
$\omega^\star \in \Omega^{(2^m \; k)}$ be such 
that $\overline{\omega^\star} = \pm \, \overline{\omega g}$.
The $p$-cycles $z_{j_1}^g$, \ldots, $z_{j_{2t}}^g$
correspond precisely to the orbits of $R_r$ contained in $\supp \mathcal{I}(\omega^\star)$.
Hence 
$\overline{\omega^\star} \in \mathcal{A}_{2t}$,
and so the vector space $\langle \mathcal{A}_{2t} \rangle$ %_F$
is invariant under $g$. 
Since $\mathcal{A} = \bigcup_{t \in T_r} \mathcal{A}_{2t}$
the lemma follows.
\end{proof}

There is an obvious factorization
$N_{S_{2m+k}}(R_r) = N_{S_{rp}}(R_r) \times S_{\{rp+1,\ldots, 2m+k\}}$.
The next proposition establishes a corresponding tensor
factorization of the $N_{S_{2m+k}}(R_r)$-module
$\langle \mathcal{A}_{2t} \rangle$. The 
shift required to make the second factor $H^{(2^{m-tp} \; k-(r-2t)p)}$ a module
for $FS_{\{rp+1,\ldots, 2m+k\}}$ is made explicit in the proof.

\begin{proposition}\label{prop:Rcorr}
If 
$t \in T_r$ then
there is an isomorphism 
\[ \langle \mathcal{A}_{2t} \rangle \cong % _F \cong
H^{(2^{tp} \; (r-2t)p)}(R_{r}) \boxtimes H^{(2^{m-tp} \; k-(r-2t)p)}
\]
of $F(N_{S_{rp}}(R_r) \times S_{\{rp+1, \ldots, 2m+k\}})$-modules.
\end{proposition}

\newcommand{\+}{{\raisebox{0.5pt}{$\scriptscriptstyle +$}}}

\begin{proof}
In order to simplify the notation we shall write $K$ for 
the $FS_{2m+k}$-submodule $K^{(2^{m} \; k)}$ of 
$F\Delta^{(2^{m} \; k)}$
defined just before Lemma~\ref{lemma:DeltaKIso}. Recall
that if $\omega \in \Omega^{(2^{m} \; k)}$ 
then, by  definition, $\overline{\omega} =
\omega + K$.
Let $J = K^{(2^{tp} \; (r-2t)p)}$. 
It follows from Proposition~\ref{prop:Brauer}, in the same way as in
 Lemma~\ref{lemma:Rcorr},
 that
$H^{(2^{tp} \; (r-2t)p)}(R_r)$
has as a basis
\[ \{\omega + J : \omega \in \Omega^{(2^{tp} \;
(r-2t)p)},\ \mathcal{I}(\omega) \in \C_{S_{rp}}(R_r)\}. \]
Define $\Delta^\+$ by shifting  the entries in each of the elements
of $\Delta^{(2^{m-tp},k-(r-2t)p)}$ by $rp$, so that
$F\Delta^\+$ is an $FS_{\{rp+1,\ldots, 2m+k\}}$-module,
and similarly
define $\Omega^\+ \subseteq \Delta^\+$ by
shifting %the elements of 
$\Omega^{(2^{m-tp}, k-(r-2t)p)}$
and  $J^\+ \subseteq F\Delta^\+$ 
by shifting the basis elements of $K^{(2^{m-tp},k-(r-2t)p)}$. Then, by
Lemma~\ref{lemma:DeltaKIso}, $H^\+ = F\Delta^\+ / J^\+$ 
is an $FS_{\{rp+1,\ldots, 2m+k\}}$-module
with basis 
%\[ 
\[ \{ \omega^\+ + J^\+ : \omega^\+ \in \Omega^\+\}. \] % \]

We shall  define a linear map
$f : \langle \mathcal{A}_{2t} \rangle \rightarrow %_F \rightarrow
H^{(2^{tp} \; (r-2t)p)} \boxtimes H^\+$. 
Given $\omega + K \in \mathcal{A}_{2t}$ where 
\[ \omega =  \bigl\{ \{i_1, i'_1\}, \ldots, \{i_m,i'_m\}, (j_1, \ldots, j_k )\}
 \in \Omega^{(2^m \; k)} 
\]
and the notation is chosen so that %$j_1 < \ldots < j_k$ and
\[ \{ i_1, i_1', \ldots, i_{tp}, i'_{tp},
j_1, \ldots, j_{(r-2t)p} \} = \{1,\ldots,rp \}, \]
we define
$(\omega + K)f = (\alpha + J) \otimes (\alpha^\+ + J^\+)$ 
where 
\begin{align*}
\alpha  &= \bigl\{ \{i_1, i'_1\}, \ldots, \{i_{tp},i'_{tp}\},
(j_1,\ldots,j_{(r-2t)p}) \bigr\} \\
\alpha^\+ &= \bigl\{ \{i_{tp+1}, i'_{tp+1}\}, \ldots, \{i_{mp},i'_{mp}\},
(j_{(r-2t)p+1},\ldots,j_k) \bigr\} .
\end{align*}
This defines a bijection between
$\mathcal{A}_{2t}$
and the basis for
$H^{(2^{tp} \; (r-2t)p)}(R_r) \boxtimes H^\+$ 
afforded by the bases for $H^{(2^{tp} \; (r-2t)p)}(R_r)$ and $H^\+$ 
just defined.
The map~$f$ is therefore a well-defined linear isomorphism.

Suppose that $\omega \in \Omega^{(2^m \; k)}$ is as above and let $g \in N_{S_{2m+k}}(R_r)$.
Let $h \in S_{\mathcal{T}(\omega g)}$ be the unique permutation such
that $(j_1 gh, \ldots, j_k gh)$ is increasing.
Let $\omega^\star = \omega g h$, %,
so $\omega^\star \in \Omega^{(2^m \; k)}$ and
$\overline{\omega} g = \sgn(h) \overline{\omega^\star}$.
Since
$g$ permutes $\{1,\ldots,rp\}$ we may factorize $h$ as $h=xx^\+$ where
$x \in S_{\mathcal{T}(\alpha)}$ and $x^\+ \in S_{\mathcal{T}(\alpha^\+)}$. 
By definition of $f$ we have
\[ (\omega^\star + K) f = (\alpha gx + J) \otimes (\alpha^\+ gx^\+ + J^\+). \]
Hence
\begin{align*}  (\omega + K)gf &= \sgn(h) (\omega^\star + K)f  \\
&= \sgn(h)\sgn(x)\sgn(x^\+) (\alpha g + J) \otimes (\alpha^+g + J^\+)  \\ &=
(\omega + K)f g.
\end{align*}
The map $f$ is therefore a homomorphism of $FN_{S_{2m+k}}(R_r)$-modules. Since 
$f$ is a linear isomorphism, the proposition follows.
\end{proof}

\subsection*{Second step:~the vertex of $H^{(2^{tp} \; (r-2t)p)}(R_r)$}
Fix $r \in \N$ and $t \in \N_0$ such that $2t \le r$.
In the second step  we
show that the $FN_{S_{rp}}(R_r)$-module $H^{(2^{tp} \; (r-2t)p)}(R_r)$
is indecomposable and that it has the subgroup $Q_t$ defined below
as a vertex.

To simplify the notation, we denote $H^{(2^{tp}\; (r-2t)p)}(R_r)$ by $M$.
Let $C$ and~$E_t$ be the elementary abelian $p$-subgroups of $N_{S_{rp}}(R_r)$
defined by
\begin{align*} 
C &= \langle z_1 \rangle \times \langle z_2 \rangle \times 
\cdots \times\langle z_r \rangle, \\
E_t &= \langle z_1z_{t+1} \rangle \times \cdots
\times \langle z_{t}z_{2t} \rangle \times \langle
z_{2t+1} \rangle \times \cdots \times \langle z_r \rangle,
\end{align*} 
where the $z_j$ are the $p$-cycles defined at the start of this section.
For $i \in \{1,\ldots, tp\}$, let $i' = i+tp$, and for $g \in S_{\{1,\ldots, tp\}}$,
let $g' \in S_{\{tp+1,\ldots, 2tp\}}$ be the permutation defined by $i'g' = (ig)'$.
Note that if $1 \le j \le t$ then $z_j' = z_{j+t}$.
Let $L$ be the group consisting of all permutations $gg'$ 
where~$g$ lies in a Sylow $p$-subgroup
of $S_{\{1,\ldots, tp\}}$ with base group $\langle z_1, \ldots, z_t \rangle$, chosen 
so that $z_1\cdots z_t$ is in its centre.
Let $L^\+$ be a Sylow $p$-subgroup
of $S_{\{2tp+1,\ldots, rp\}}$ with base group $\langle z_{2t+1}, \ldots, z_r \rangle$,
chosen so that $z_{2t+1}\cdots z_r$ is in its centre. 
(The existence of such Sylow $p$-subgroups follows from
the construction of Sylow $p$-subgroups of symmetric groups
as iterated wreath products in \cite[4.1.19 and 4.1.20]{JK}.)
Let
\[ Q_t = L \times L^\+. \]
Observe that $Q_t$ normalizes $C$ and so $\langle C, Q_t \rangle$
is a $p$-group contained in $\C_{S_{rp}}(R_r)$. 
Let $P$ be a Sylow $p$-subgroup of $\C_{S_{rp}}(R_r)$
containing $\langle Q_t, C\rangle$. 
Since there is a Sylow $p$-subgroup of~$S_{rp}$ containing
$R_r$ in its centre,~$P$ is also a Sylow $p$-subgroup of $S_{rp}$.
Clearly $E_t \le C$ and
\[ R_r \le E_t \le Q_t \le P \le \C_{S_{rp}}(R_r). \]

If $t=0$ then $M$
is the sign representation of $N_{S_{rp}}(R_r)$, with $p$-permutation
basis $\mathcal{B} = \{ \omega \}$ where $\omega$ is the unique
element of $\Omega^{(2^0 \; rp)}$. It is then clear
that $M$ has the Sylow $p$-subgroup~$Q_0$ of $\C_{S_{rp}}(R_r)$ as a vertex.
We may therefore assume that $t \in \N$ for the rest of this step.

By Proposition~\ref{prop:Brauer} there is a choice of signs
$s_\omega \in \{+1,-1\}$ for $\omega \in \Omega^{(2^{tp} \; (r-2t)p)}$ such that
\[ \mathcal{B} = \{s_\omega \overline{\omega} : \omega \in \Omega^{(2^{tp} \;
(r-2t)p)}, \  
\mathcal{I}(\omega) \in \C_{S_{rp}}(R_{r}) \} \]
is a $p$-permutation basis for $M$ with respect to $P$.
Let 
\[ \mathcal{O}_j = \{(j-1)p+1,\ldots, jp \} \] 
be the orbit
of $z_j$ on $\{1,\ldots,rp\}$ of length $p$. 
If $\mathcal{I}(\omega) \in \C_{S_{rp}}(R_{r})$
then
$\mathcal{I}(\omega)$ permutes these
orbits as blocks for its action; let 
\[ \mathcal{I}_\mathcal{O}(\omega) \in S_{\{ \mathcal{O}_1, \ldots, \mathcal{O}_r \}} \]
be the involution induced by the action of $\mathcal{I}(\omega)$
on the set of orbits. 
%Note that if 
%$t \in \N$ then $\mathcal{I}_\mathcal{O}(\omega)$ is an involution and
%if $t=0$ then $\mathcal{I}_\mathcal{O}(\omega)$ is the identity permutation.

\begin{proposition}\label{prop:Evertex}
The $FN_{S_{rp}}(R_r)$-module $M$ is indecomposable and has a vertex containing $E_t$.
\end{proposition}

\begin{proof}
For each involution
$h \in S_{\{\mathcal{O}_1,\ldots,\mathcal{O}_r\}}$ 
that fixes exactly $r-2t$ of the orbits $\mathcal{O}_j$, and so moves the other $2t$, define
\[ \mathcal{B}(h) = \{ s_\omega \overline{\omega} \in \mathcal{B} :
\mathcal{I}_\mathcal{O}(\omega) = h \}. \]
Clearly there is a vector space decomposition
\[  M = \bigoplus_h \langle \mathcal{B}(h) \rangle. \]
If $g \in C$ then $\mathcal{I}_{\mathcal{O}}(\omega g) = \mathcal{I}_{\mathcal{O}}(\omega)$ since
$g$ acts trivially on the set of orbits $\{ \mathcal{O}_1, \ldots, \mathcal{O}_r\}$.
Therefore~$C$ permutes the elements of each $\mathcal{B}(h)$.  

Let 
\[ h^\star = (\mathcal{O}_1, \mathcal{O}_{t+1}) \cdots (\mathcal{O}_{t},
\mathcal{O}_{2t} ) \in S_{\{\mathcal{O}_1,\ldots,\mathcal{O}_r\}} \]
and let $s_{\omega^\star} \overline{\omega^{\star}} \in \mathcal{B}(h^\star)$ be the 
unique basis element such that 
\[ \mathcal{I}(\omega^\star) = (1, tp+1)(2, tp+2)
\cdots (tp,2tp)\]
(Equivalently, $\mathcal{I}(\omega^\star)$ is the unique involution in $S_{2tp}$
that preserves
the relative orders of the elements in $\mathcal{O}_j$ for $1\le j \le 2t$
and satisfies $\mathcal{I}_\mathcal{O}(\omega^\star) = h^\star$.)
By Lemma~\ref{lemma:PPermBasis}(ii) we
see that the stabiliser of $\overline{\omega^\star}$ in~$C$ is the subgroup~$E_t$.
%EG
Let $s_\delta\overline{\delta}\in B(h^\star)$, 
then $\mathcal{I}(\delta)=\mathcal{I}(\omega^\star)=h^\star.$
Without loss of generality we have that $$\mathcal{I}(\delta)=(1, i_1)(2, i_2)\cdots (tp, i_{tp}),$$
where $\{i_{(j-1)p+1},i_{(j-1)p+2},\ldots, i_{jp}\}=\mathcal{O}_{t+j}$ for all 
$j\in\{1,\ldots,t\}$.
Hence, there exist $k_1,k_2,\ldots, k_t\in\{0,1,\ldots,p-1\}$ and a permutation 
\[ g=z_{t+1}^{k_1}z_{t+2}^{k_2}\cdots z_{2t}^{k_t}, \]
such that 
$(tp+(j-1)p+1)g = i_{(j-1)p+1}$
%$(rp+1)\rho=i_{rp+1}$ 
for all 
%$r\in\{1,2,\ldots,2t-1\}$.
$j \in \{1,\ldots, t\}$.
Since $s_\delta\overline{\delta}$ is fixed by $R_r$, it follows that
 $s_{\omega^\star}\overline{\omega^\star}g=s_\delta\overline{\delta}$.
Therefore any basis element in $\mathcal{B}(h^\star)$ can be obtained from~$\omega^\star$ by permuting
the members of $\mathcal{O}_{t+1}$, \ldots,~$\mathcal{O}_{2t}$ by an element of $C$. %EGend
It follows that $\mathcal{B}(h^\star)$ has size $p^t$ and is equal to 
the orbit of $s_{\omega^\star} \overline{\omega^\star}$ under~$C$.
Therefore
there is an isomorphism of $FC$-modules $\langle \mathcal{B}(h^\star) \rangle
\cong  F\!\ind_{E_t}^C$. By Lemma~\ref{lemma:perm}, $\langle \mathcal{B}(h^\star)
\rangle$ is an indecomposable $FC$-module with vertex~$E_t$.

For each involution $h \in S_{\{\mathcal{O}_1,\ldots,\mathcal{O}_r\}}$, the 
$FC$-submodule $\langle \mathcal{B}(h) \rangle$ of~$M$ is sent to
$\langle \mathcal{B}(h^\star) \rangle$
by an element of~$N_{S_{rp}}(R_r)$ normalizing~$C$. 
It follows that if~$U$ is any
summand of~$M$, now considered as an $FN_{S_{rp}}(R_r)$-module,
then the restriction of $U$ to~$C$ is isomorphic to a direct sum
of indecomposable $p$-permutation $FC$-modules with vertices conjugate
in~$N_{S_{rp}}(R_r)$ to~$E_t$. Applying Theorem~\ref{thm:BrauerCorr}
to these summands, % restriction of $U$ to $C$, 
we see that
there exists $g \in N_{S_{rp}}(R_r)$ such that $U(E_t^g) \not= 0$.
Now by Theorem~\ref{thm:BrauerCorr}, this time applied to
the $FN_{S_{rp}}(R_r)$-module $U$, we see that~$U$ has a
vertex containing~$E_t^g$. Hence every indecomposable summand of $M$
has a vertex containing~$E_t$.

We now calculate the Brauer correspondent $M(E_t)$.
Let $s_\omega \overline{\omega} \in \mathcal{B}$. It follows
from Lemma~\ref{lemma:PPermBasis}(ii)
that $\overline{\omega}$ is fixed by $E_t$
if and only if $\mathcal{I}_\mathcal{O}(\omega)$ is 
the involution~$h^\star$. Hence, by Corollary~\ref{cor:Brauer}
and Lemma~\ref{lemma:PPermBasis},
we have $M(E_t) = \langle \mathcal{B}(h^\star) \rangle$. 
We have already seen that $\langle \mathcal{B}(h^\star)\rangle$ 
is indecomposable as an $FC$-module. Since~$C$ normalizes
$E_t$ and centralizes $R_r$, it follows that~$M(E_t)$ is indecomposable as a module
for the normalizer of~$E_t$ in~$N_{S_{rp}}(R_r)$. 
We already know that every indecomposable summand of $M$
has a vertex containing $E_t$, so
it 
follows from Corollary~\ref{cor:Brauer} that $M$ is indecomposable.
\end{proof}

Note that if $\omega^\star$ is as defined in the proof of Proposition~\ref{prop:Evertex}, then
$Q_t$ is a Sylow $p$-subgroup of $\C_{S_{rp}}\bigl(\mathcal{I}(\omega^\star)\bigr)\cong (S_2\wr S_{tp})\times S_{(r-2t)p}$.
Using this observation and
the $p$-permutation basis $\mathcal{B}$ for~$M$ it is now straightforward to
prove the following proposition.
%show that $Q_t$ is a vertex of $m$.

\begin{proposition}\label{prop:Qvertex}
The indecomposable $FN_{S_{rp}}(R_{r})$-module $M$ has $Q_t$ as a vertex.
\end{proposition}

\begin{proof}
By Corollary~\ref{cor:Brauer}, if $Q$ is subgroup
of $P$ maximal subject to $\mathcal{B}^Q \not= \varnothing$ then $Q$ is a vertex of $M$.
By Lemma~\ref{lemma:PPermBasis}(ii), 
a basis element $s_{\omega} \overline{\omega} \in B$ is fixed by a $p$-subgroup $Q$
of $P$ if and only if
$Q \le \C_{S_{rp}}\bigl( \mathcal{I}(\omega) \bigr)$. 
Taking $\omega = \omega^\star$ we see that there is a vertex of $M$ containing $Q_t$.
On the other hand, $\C_{S_{rp}}\bigl( \mathcal{I}(\omega) \bigr)$ is conjugate
in $S_{rp}$ to $\C_{S_{rp}}\bigl( \mathcal{I}(\omega^\star) \bigr)$, and so
if $Q \le \C_{S_{rp}}\bigl( \mathcal{I}(\omega) \bigr)$ then
$|Q| \le |Q_t|$. It follows that $Q_t$ is a vertex of $M$.
\end{proof}

\subsection*{Third step:~proof of Theorem~\ref{thm:vertices}}
For the remainder of the proof we shall %abuse notation and 
regard $S_{(r-2t)p}$ as acting
on $\{2tp+1,\ldots, rp\}$. We  denote by $D_t$ the $p$-group $C\cap N_{S_{rp}}(Q_t)$.
Notice that $\langle D_t, Q_t\rangle$ is a $p$-group since is a subgroup of $\langle C, Q_t\rangle\leq P$.
We shall need the following lemma to work with modules for $N_{S_{rp}}(Q_t)$.

\begin{lemma}\label{lemma:Mark1}
The unique Sylow $p$-subgroup of $N_{S_{rp}}(Q_t)$ is the subgroup
$\langle D_t, Q_t \rangle$ of $P$.
\end{lemma}

\begin{proof}
Let $x \in N_{S_{rp}}(Q_t)$. If $2t+1 \le j \le r$ then the conjugate $z_j^x$
of the $p$-cycle $z_j \in E_t$ is a $p$-cycle in $Q_t$.
Since $Q_t$ normalizes $E_t$, it permutes the orbits $\mathcal{O}_1$, \ldots, $\mathcal{O}_r$
of $E_t$ as blocks for its action.
 No $p$-cycle can
act non-trivially on these blocks, so %, as before, we see that 
$z_j^x \in 
\langle z_{2t+1}, \ldots, z_r \rangle$.
Hence %we see that 
if $1 \le j \le t$ then $(z_{j}z_{j+t})^x \in \langle
z_1z_{t+1}, \ldots, z_{t}z_{2t} \rangle$.
It follows that $N_{S_{rp}}(Q_t)$ factorizes as
\[ N_{S_{rp}}(Q_t) = N_{S_{2tp}}(\L) \times N_{S_{(r-2t)p}}(L^+) \]
where $L$ and $L^\+$ are as defined at the start of the second step.
Moreover, we see that $N_{S_{rp}}(Q_t)$ permutes, as blocks for
its action, the sets 
$\mathcal{O}_1 \hskip1pt\cup\hskip1pt \mathcal{O}_{t+1}, \ldots, 
\mathcal{O}_t \hskip1pt\cup\hskip1pt \mathcal{O}_{2t}$
and $\mathcal{O}_{2t+1}, \ldots, \mathcal{O}_{r}$.

Let $h \in N_{S_{rp}}(Q_t)$ be a $p$-element.
We may factorize $h$ as
$gg^\+$ where $g \in N_{S_{2tp}}(\L)$ and $g^\+ \in N_{S_{(r-2t)p}}(L^\+)$ are $p$-elements.
Since $\langle L^\+, g^\+\rangle$ is a $p$-group and
$L^\+$ is a Sylow $p$-subgroup of 
$S_{(r-2t)p}$, we have $g^\+ \in L^\+$. 
Let
\[ X = \{ \mathcal{O}_1 \cup \mathcal{O}_{t+1}, \ldots, \mathcal{O}_{t}
\cup \mathcal{O}_{2t} \}. \]
The group
$\langle L, g \rangle$ permutes the sets in $X$ as blocks for its action. Let
\[ \pi : \langle \L, g \rangle \rightarrow S_X \]
be the corresponding group homomorphism. 
By construction $L$ acts on the sets $\mathcal{O}_1, \ldots, \mathcal{O}_t$
as a Sylow $p$-subgroup of $S_{\{ \mathcal{O}_1, \ldots, \mathcal{O}_t \}}$; hence
$L \pi$ is a Sylow $p$-subgroup of $S_X$.
Since $\langle \L, g\rangle$
is a $p$-group, there exists $\tilde{g} \in L$ such that $g \pi = \tilde{g} \pi$. Let $y =
g\tilde{g}^{-1}$. Since $y$ acts trivially on $X$, we may write
\[ y = g_1 \ldots g_t \]
where $g_j \in S_{\mathcal{O}_{j} \hskip1pt\cup\hskip1pt \mathcal{O}_{j+t}}$ for each $j$.
The $p$-group $\langle L, y \rangle$ has as a subgroup $\langle z_{j}z_{j+t}, y 
\rangle$. The permutation group induced by this subgroup on $\mathcal{O}_{j}
\cup \mathcal{O}_{j+t}$, namely $\langle z_{j}z_{j+t}, g_j\rangle$,
is a $p$-group acting on a set of size $2p$. Since $p$ is odd, the unique Sylow
$p$-subgroup of $S_{\mathcal{O}_{j} \hskip1pt\cup\hskip1pt \mathcal{O}_{j+t}}$ 
containing $z_{j}z_{j+t}$
is $\langle z_{j}, z_{j+t} \rangle$. Hence $g_j \in \langle z_{j}, z_{j+t} \rangle$ for
each $j$. Therefore $y \in \langle z_1, \ldots, z_{t}, z_{t+1}, \ldots, z_{2t} \rangle \le C$.
We also know that $y\in \langle Q_t,g\rangle\leq N_{S_{2tp}}(Q_t)\leq N_{S_{rp}}(Q_t)$. Therefore $y\in D_t$, 
and 
since $\tilde{g} \in Q_t$, it follows that $g \in \langle D_t, Q_t \rangle$.
Hence $h = gg^\+ \in \langle D_t, Q_t \rangle\leq \langle C, Q_t\rangle\leq P$.

Conversely, the subgroup $\langle D_t, Q_t\rangle$ is contained in $N_{S_{rp}}(Q_t)$ because both $D_t$ and $Q_t$ are.
It follows that $\langle D_t, Q_t \rangle$ is the unique Sylow $p$-subgroup of $N_{S_{rp}}(Q_t)$.
\end{proof}

We also need the following two general lemmas.

\begin{lemma}\label{lemma:doubleBrauer}
Let $Q$ and $R$ be $p$-subgroups of a finite group $G$ and let 
$U$ be a $p$-permutation $FG$-module. Let $K = N_G(R)$.
If $R$ is normal in $Q$ then the Brauer correspondents $U(Q)$ and 
$\bigl( U(R) \bigr)(Q)$ are isomorphic as $FN_{K}(Q)$-modules.
\end{lemma}

\begin{proof}
Let $P$ be a Sylow $p$-subgroup of $N_G(R)$ containing $Q$
and let $\mathcal{B}$ be a $p$-permutation basis for $U$ with respect to 
a Sylow $p$-subgroup of $G$ containing~$P$.
By Corollary~\ref{cor:Brauer} we have $U(Q) = \langle \mathcal{B}^Q \rangle$ as an 
$FN_G(Q)$-module. 
In particular 
\[ U(Q)\res_{N_K(Q)}\! {}={} \langle \mathcal{B}^Q \rangle \]
as an $FN_{K}(Q)$-module.
On the other hand $U(R) = \langle \mathcal{B}^R \rangle$ as an $FN_G(R)$-module. Now $\mathcal{B}^R$ is a
$p$-permutation basis for~$U(R)$ with respect to $K \cap P$.
Since this subgroup contains $Q$ we have $\bigl( U(R) \bigr) (Q) =
\langle \mathcal{B}^R \rangle (Q) = \langle (\mathcal{B}^R)^Q \rangle
= \langle \mathcal{B}^Q \rangle$ as $FN_{K}(Q)$-modules, as required.
\end{proof}

\begin{lemma}\label{lemma:tensorBrauer}
Let $G$ and $G'$ be finite groups and let $U$ and $U'$ be $p$-permutation
modules for $FG$ and $FG'$, respectively. 
If $Q \le G$ is a $p$-subgroup then $(U \boxtimes U')(Q) = U(Q) \boxtimes U'$,
where on the left-hand side
$Q$ is regarded as a subgroup of $G \times G'$ in the obvious way.
\end{lemma}

\begin{proof}
This follows from Corollary~\ref{cor:Brauer} by taking
$p$-permutation bases for $U$ and $U'$ such that the $p$-permutation basis for $U$
is permuted by a Sylow $p$-subgroup of $G$ containing~$Q$.
\end{proof}

We are now ready to prove Theorem~\ref{thm:vertices}. We repeat the
statement below for the reader's convenience.

\setcounter{section}{1}
\setcounter{theorem}{1}
\begin{theorem}
\thmVertices
\end{theorem}
\setcounter{section}{4}

\begin{proof} %[Proof of Theorem~\ref{thm:vertices}]
Let $r\in \N$ be maximal such that the subgroup $R_r$ is contained
in a vertex of $U$.  Let $K = N_{S_{rp}}(R_r)$.
By Lemma~\ref{lemma:Rcorr} and
Proposition~\ref{prop:Rcorr} there is an isomorphism of $N_{S_{2m+k}}(R_r)$-modules
\[ H^{(2^m\; k)}(R_r) \cong \bigoplus_{t \in T_r} 
\big(H^{(2^{tp}\; (r-2t)p)}(R_r) \boxtimes H^{(2^{m-tp}\; k-(r-2t)p)}\big) \]
compatible with the factorization $N_{S_{2m+k}}(R_r) = K %N_{S_{rp}}(R_r)
\times S_{2m+k-rp}$, where we regard $S_{2m+k-rp}$ as acting on
$\{rp+1,\ldots,2m+k\}$ and shift each module $H^{(2^{m-tp} \; k-(r-2t)p)}$ appropriately.

For $t \in T_r$, let $M_t = H^{(2^{tp}\; (r-2t)p)}(R_r)$.
By Proposition~\ref{prop:Qvertex}, each $M_t$ 
is indecomposable as an $FN_{S_{rp}}(R_r)$-module.
% and has vertex $Q_t\in Syl_p(S_2\wr S_{tp}\times S_{(r-2t)p})$. 
Hence, by the Krull--Schmidt Theorem, there is a subset $T' \subset T_r$, and
for each $t \in T'$, a non-zero summand~$W_t$ of $H^{(2^{m-tp}\; k-(r-2t)p)}$,
such that
\[ U(R_r) \cong \bigoplus_{t \in T'} M_t \boxtimes W_t
\]
as $F(K \times S_{2m+k-rp})$-modules. By Proposition~\ref{prop:Qvertex},
$M_t$ has $Q_t$ as a vertex for each non-zero $t \in T'$. 
It is clear that $M_0 = \sgn_{S_{rp}}(R_r)$ has vertex $Q_0$ as an $FN_{S_{rp}}(R_r)$-module.
Let $\ell$ be the least element of $T'$.
If $t > \ell$ then $Q_t$ does not contain a conjugate
of the subgroup $E_\ell$ of $Q_\ell$. Hence, by Theorem \ref{thm:BrauerCorr},
we have $M_t(Q_\ell) = 0$.
It now follows from Lemmas~\ref{lemma:doubleBrauer} and~\ref{lemma:tensorBrauer} 
that there is an isomorphism of $F(N_K(Q_\ell) \times S_{2m+k-rp})$-modules
\[ U(Q_\ell) \cong U(R_r)(Q_\ell) \cong M_\ell(Q_\ell) \boxtimes W_\ell. \]
Since $M_\ell$ has $Q_\ell$ as a vertex, we have $M_\ell(Q_\ell) \not=0$. It follows that
$U$ has a vertex $Q$ containing $Q_\ell$.

Let $\mathcal{B}$ be the $p$-permutation basis for $M_\ell$ %$H^{(2^{\ell p} \; (r-2\ell)p)}(R_r)$
defined in the second step. 
Since $\mathcal{B}$ is permuted by the Sylow $p$-subgroup $P$ of $K$,
it follows from Corollary~\ref{cor:Brauer} 
and Lemma~\ref{lemma:Mark1} that $\mathcal{C} = \mathcal{B}^{Q_\ell}$ is a 
$p$-permutation basis for the $FN_K(Q_\ell)$-module $M_\ell(Q_\ell)$
with respect to the Sylow $p$-subgroup
$\langle D_\ell, Q_\ell \rangle$ of $N_K(Q_\ell)$. 
Since $W_\ell$ is isomorphic to a direct summand of the $p$-permutation module
$H^{(2^{m-\ell p} \; k-(r-2\ell)p)}$ it has a $p$-permutation basis~$\mathcal{C}^\+$
with respect to a Sylow $p$-subgroup $P^\+$ of $S_{\{rp+1,\ldots, 2m+k\}}$.
Therefore 
\[ \mathcal{C} \boxtimes \mathcal{C}^\+ = \{ v \otimes v^\+ : v \in \mathcal{C},
v^\+ \in \mathcal{C}^\+ \} \]
is a $p$-permutation basis for $M_\ell(Q_\ell) \boxtimes W_\ell$ with
respect to the Sylow subgroup $\langle D_\ell, Q_\ell \rangle
\times P^+$ of $N_K(Q_\ell) \times S_{2m+k-rp}$.

Suppose, for a contradiction, that $Q$ strictly contains $Q_\ell$. 
Since~$Q$ is a $p$-group there exists a $p$-element
$g \in N_{Q}(Q_\ell)\leq N_{S_{2m+k}}(Q_\ell)$ such that
$g \not\in Q_\ell$. Now $Q_\ell$ has orbits of length at least $p$
on $\{1,\ldots, rp\}$ and fixes $\{rp+1,\ldots, 2m+k\}$. Since $g$
permutes these orbits as blocks for its action, we may factorize
$g$ as $g = hh^\+$ where $h\in N_{S_{rp}}(Q_\ell)$ and %$h \in S_{\{1,\ldots,rp\}}$ and 
$h^\+ \in S_{2m+k-rp}$. %\{rp+1,\ldots, 2m+k\}}$.
By Lemma \ref{lemma:Mark1} we have that $\langle Q_\ell, h \rangle \le N_K(Q_\ell)$. 

Corollary~\ref{cor:Brauer} now implies that
$(\mathcal{C} \boxtimes \mathcal{C}^\+)^{\langle Q_\ell, g\rangle}
\not=\varnothing$. Let \hbox{$v \otimes v^\+ \in \mathcal{C} \boxtimes \mathcal{C}^\+$}
be such that $(v \otimes v^\+)g = v \otimes v^\+$.
Then $v \in \mathcal{B}^{\langle Q_\ell, h \rangle}$.
But $Q_\ell$ is a vertex
of $M_\ell$, %by Proposition~\ref{prop:Qvertex}, 
so it follows from
Corollary~\ref{cor:Brauer} that $h \in Q_\ell$. Hence $h^\+$ is a
non-identity element of~$Q$. By taking an appropriate power
of $h^\+$ we find that~$Q$ contains a product of one or more $p$-cycles
with support contained in \hbox{$\{rp+1, \ldots, 2m+k\}$}. This contradicts
our assumption that $r$ was maximal such that $R_r$ is contained
in a vertex of $U$.

Therefore $U$ has vertex $Q_\ell$. 
We saw above that there is an isomorphism $U(Q_\ell) \cong M_\ell(Q_\ell)
\boxtimes W_\ell$ of $F(N_K(Q_\ell) \times S_{2m+k-rp})$-modules.
This identifies $U(Q_\ell)$ as a vector space on which $N_{S_{2m+k}}(Q_\ell)
=  N_{S_{rp}}(Q_\ell) \times S_{2m+k-rp}$ 
acts. It is clear from the isomorphism in Proposition~\ref{prop:Rcorr}
that $N_{S_{rp}}(Q_\ell)$ acts on the first tensor factor and $S_{2m+k-rp}$
acts on the second. Hence the action of $N_K(Q_\ell)$ on $M_\ell(Q_\ell)$
extends to an action of $N_{S_{rp}}(Q_\ell)$ on $M_\ell(Q_\ell)$
and we obtain a tensor factorization $V \boxtimes W_\ell$ of $U(Q_\ell)$
as a $N_{S_{rp}}(Q_\ell) \times S_{2m+k-rp}$-module. %an $N_{S_{2m+k}}(Q_\ell)$-module.
An outer tensor product of modules is projective if and only if
both factors are projective,
so by Theorem~\ref{thm:BrauerBij},
$V$ 
%$M_\ell(Q_\ell)$ 
is a projective $FN_{S_{rp}}(Q_\ell)/Q_\ell$-module,
$W_\ell$ is a projective $FS_{2m+k-rp}$-module, and
$U(Q_\ell)$ is the Green correspondent of~$U$.
\end{proof}

\section{Proofs of  Theorem~\ref{thm:main} and Proposition~\ref{prop:projectives}}
In this section we prove Proposition~\ref{prop:projectives},
and hence Theorem~\ref{thm:main}. It will be convenient to 
assume that $H^{(2^m \; k)}$
is defined over the finite field~$\F_p$. 
Proposition~\ref{prop:projectives} then follows for an arbitrary
field of characteristic $p$ by change of scalars. We assume
the common hypotheses for these results, 
so~$\gamma$ is a $p$-core such that $2m+k = |\gamma| + w_k(\gamma)p$ and if $k \ge p$ then
\[ w_{k-p}(\gamma) \not= w_{k}(\gamma)-1.  \]
 Let $\lambda$ be
a maximal element of $\mathcal{E}_k(\gamma)$ under the 
dominance order.

Write $H_\Q^{(2^m \; k)}$ for the twisted Foulkes module
defined over the rational field. This module has an ordinary
character given by Lemma~\ref{lemma:FoulkesChars}.
In particular it has $\chi^\lambda$ as a constituent, and so the
rational Specht module~$S_\Q^\lambda$ is a direct summand of $H_\Q^{(2^m\; k)}$.
Therefore, by reduction modulo~$p$, each composition factor of $S^\lambda$
(now defined over $\F_p$) appears in $H^{(2^m \; k)}$.
In particular $H^{(2^m \; k)}$ has non-zero block component 
for the block $B(\gamma, w_k(\gamma))$ with $p$-core $\gamma$ and weight $w_k(\gamma)$.
We now use Theorem~\ref{thm:vertices} to %and other results from \S 4 to
show that any such block component is projective.

\begin{proposition}\label{prop:allproj}
The block component of $H^{(2^m \; k)}$ 
for the block $B(\gamma, w_k(\gamma))$ of $S_{2m+k}$
is projective.
\end{proposition}

\begin{proof}
Suppose, for a contradiction,
that $H^{(2^m \; k)}$ has a non-projective summand $U$ in
$B(\gamma, w_k(\gamma))$. By Theorem~\ref{thm:vertices},
the vertex of $U$ is a Sylow subgroup~$Q_t$ of \hbox{$(S_2 \wr S_{tp})
\times S_{(r-2t)p}$} for some $r \in \N$ and $t \in \N_0$ such that $tp \le m$, $2t \le r$
and $(r-2t)p \le k$.

Suppose first of all that $2t < r$. In this case 
there is a $p$-cycle $g \in Q_t$. Replacing $Q_t$ with a conjugate,
we may assume that $g = (1,\ldots, p)$ and so $\langle g \rangle = R_1$
where $R_1$ is as defined at the start of the first step in \S 4.
By Lemma~\ref{lemma:Rcorr} and Proposition~\ref{prop:Rcorr}, we have that $k \ge p$ and $U(R_1)$
is a direct summand of
\[ H^{(2^m \; k)}(R_1) =  \sgn_{S_p}(\langle g \rangle) \boxtimes H^{(2^m \; k - p)}. \]
Hence there exists an indecomposable summand $W$ of $H^{(2^m \; k-p)}$
such that 
\[ \sgn_{S_p}(\langle g \rangle) \boxtimes W \mid  U(R_1) .\] 
By Theorem~\ref{thm:BrauerBlocks}, $W$ lies in the
block $B(\gamma, w_k(\gamma) - 1)$ of $S_{2m+k-p}$. 
In particular, this implies that $H^{(2^m\; k-p)}$ has a composition factor in this block. 
Therefore there is a constituent $\chi^\mu$ of the ordinary character 
of $H^{(2^m \; k-p)}$ such that $S^\mu$ 
lies in $B(\gamma,w_k(\gamma)-1)$. 
But then, by Lemma~\ref{lemma:FoulkesChars},~$\mu$ 
is a partition with $p$-core $\gamma$ having exactly $k-p$
odd parts and weight $w_k(\gamma) - 1$. Adding a single vertical $p$-hook
to $\mu$ gives a partition of weight $w_k(\gamma)$ with exactly $k$ odd parts.
Hence $w_{k-p}(\gamma) = w_k(\gamma) - 1$, contrary to the hypothesis on $w_{k-p}(\gamma)$.

Now suppose that $2t = r$. Let $g = (1, \ldots, p)(p+1, \ldots, 2p)$. Then
$g \in Q_t$ by definition and $\langle g \rangle = R_2$. By
Lemma~\ref{lemma:Rcorr} and Proposition~\ref{prop:Rcorr} we have that $U(R_2)$ 
is a direct summand of
\[ H^{(2^m \; k)}(R_2) = \Bigl( H^{(2^p)}(\langle g \rangle) \boxtimes H^{(2^{m-p} \; k)} 
\Bigr) \bigoplus
\Bigl( \sgn_{S_{2p}}(\langle g \rangle) \boxtimes H^{(2^m \; k-2p)} \Bigr) \]
where the second summand should be disregarded if $k < 2p$.
It follows that either there is an $FS_{2m+k-2p}$-module $V$ such that
\[ H^{(2^p)}(\langle g \rangle) \boxtimes V \mid  U(R_2),\] 
or $k \ge 2p$ and there is an $FS_{2m+k-2p}$-module $W$ such that 
\[ \sgn_{S_{2k}}(\langle g \rangle) 
\boxtimes W \mid  U(R_2). \] 

Again we use Theorem~\ref{thm:BrauerBlocks}.
In the first case the theorem implies that
$V$ lies in the block \hbox{$B(\gamma, w_k(\gamma)-2)$} of $S_{2m+k-2p}$.
Hence there is a constituent $\chi^\mu$ of the ordinary character of $H^{(2^{m-p} \; k)}$
such that $\mu$ is a partition with $p$-core $\gamma$ and weight $w_k(\gamma) - 2$ 
having exactly $k$ odd parts.
This contradicts the minimality of $w_k(\gamma)$. In the second case
$W$ also lies in the block $B(\gamma, w_k(\gamma) - 2)$ of $S_{2m+k-2p}$
and there is a constituent $\chi^\mu$ of the ordinary character of $H^{(2^m \; k-2p)}$
such that $\mu$ is a partition with $p$-core~$\gamma$ and weight $w_k(\gamma) - 2$
having exactly $k-2p$ odd parts. But then by adding a single vertical $p$-hook
to $\mu$ we reach a partition with weight $w_k(\gamma) - 1$ having exactly
$k-p$ odd parts. Once again this
contradicts the hypothesis that $w_{k-p}(\gamma)\not= w_k(\gamma)-1$.
\end{proof}

For $\nu$ a $p$-regular partition,
let $P^\nu$ denote the projective cover of the simple module $D^\nu$.
To finish the proof of Proposition~\ref{prop:projectives} 
we must show
that if $\lambda$ is a maximal element of $\mathcal{E}_k(\gamma)$ then $P^\lambda$
is one of the projective summands of $H^{(2^m \; k)}$ in the block $B(\gamma, w_k(\gamma))$.
For this we need a lifting result for summands of the monomial
module $H^{(2^m \; k)}$, which we prove using the analogous, and 
well known, result 
for permutation modules.
Let $\Z_p$ denote the ring
of $p$-adic integers and let $H^{(2^m \; k)}_{\Z_p}$ denote the twisted
Foulkes module defined over $\Z_p$.

\begin{lemma}\label{lemma:lifting}
If $U$ is a %an indecomposable 
direct summand of $H^{(2^m\; k)}$ 
then there is a %an indecomposable 
$\Z_p S_{2n+k}$-module
$U_{\Z_p}$, unique up to isomorphism,
such that $U_{\Z_p}$ is a direct summand of $H^{(2^m \; k)}_{\Z_p}$
and $U_{\Z_p} \otimes_{\Z_p} \F_p \cong U$.
\end{lemma}

\begin{proof}
Let $A_k$ denote the alternating group on $\{2m+1,\ldots, 2m+k\}$.
Let $M = \F_p \ind_{S_2 \wr S_m \times A_k}^{S_{2m+k}}$ be the permutation
module of $S_{2m+k}$ acting on the cosets of $S_2 \wr S_m \times A_k$
and let $M_{\Z_p} = \Z_p\!\ind_{S_2 \wr S_m \times A_k}^{S_{2m+k}}$ be the
corresponding permutation module defined over $\Z_p$.
Since $p$ is odd, the trivial \hbox{$\Z_p(S_2 \wr S_m \times S_k)$} module is a direct
summand of $\Z_p\!\ind_{S_2 \wr S_m \times A_k}^{S_2 \wr S_m \times S_k}$. 
Hence,
inducing up to $S_{2m +k}$ (as in the remark
after Lemma~\ref{lemma:DeltaKIso}), we see that
$M_{\Z_p} = H^{(2^m \; k)}_{\Z_p} \oplus M'_{\Z_p}$ where $M'_{\Z_p}$ is a complementary
$\Z_pS_{2m+k}$-module, and~$M = H^{(2^m \; k)} \oplus M'$ where $M'$ is the
reduction modulo $p$ of~$M'_{\Z_p}$.

By Scott's lifting theorem (see \cite[Theorem 3.11.3]{Benson}), 
reduction modulo $p$ is a bijection between the
summands of $M_{\Z_p}$ and the summands of $M$. By the same result, this
bijection restricts to a bijection between the summands of the permutation
module~$M'_{\Z_p}$ and the summands of $M'$. Since $U$ is a direct~summand of $M$
there is a 
%Since $U$ is a direct summand of $M$ it follows from 
%\cite[\S 3.11]{Benson} that there is a
summand~$U_{\Z_p}$ of $M_{\Z_p}$, unique up to isomorphism, such
that \hbox{$U_{\Z_p} \otimes_{\Z_p} \F_p \cong U$}. By the remarks
just made, $U_{\Z_p}$ is isomorphic to a summand of $H^{(2^m \; k)}_{\Z_p}$.
%Reduction modulo $p$
%gives a bijection between the indecomposable summands of $M_{\Z_p}$ and
%the indecomposable summands of $M$. Again by \cite[\S 3.11]{Benson},
%this bijection sends summands of the permutation module $M'_{\Z_p}$ to
%summands of $M'$. Therefore reduction modulo $p$ is a bijection between the
%summands of~$H^{(2^m \; k)}_{\Z_p}$
%and the summands of $H^{(2^m \; k)}$. It follows that $U_{\Z_p}$ is isomorphic to
%a direct summand of $H^{(2^m \; k)}_{\Z_p}$.
\end{proof}

Let $P^\nu_{\Z_p}$ be the $\Z_p$-free
$\Z_pS_{2m+k}$-module whose reduction modulo $p$ is
$P^\nu$. By Brauer reciprocity (see for instance %\cite[\S2]{Navarro} or 
\cite[\S 15.4]{Serre}), the ordinary character of~$P^\nu_{\Z_p}$~is 
\[ \tag{$\star$} \psi^\nu = \sum_\mu d_{\mu \nu} \chi^\mu. \]
The result mentioned in the introduction, that 
if $d_{\mu \nu} \not=0$ then $\nu$ dominates~$\mu$, implies that
the sum may be taken over those partitions $\mu$ dominated by~$\nu$.

\begin{proof}[Proof of Proposition~\ref{prop:projectives}]
We have  seen that each summand of $H^{(2^m \; k)}$ in the block
$B(\gamma, w_k(\gamma))$ is projective and that there is at least one such summand.
Let $P^{\nu_1}, \ldots, P^{\nu_c}$ be the summands of $H^{(2^m \; k)}$
in $B(\gamma, w_k(\gamma))$. Using Lemma~\ref{lemma:lifting} to lift
these summands to summands of~$H^{(2^m \; k)}_{\Z_p}$ we see that
the ordinary character of the summand of $H^{(2^m \; k)}_{\Z_p}$
lying in the $p$-block of $S_{2m+k}$ with core $\gamma$ and weight $w_k(\gamma)$
is 
%\[ 
$\psi^{\nu_1} + \cdots + \psi^{\nu_c}$. %.\] %{P^{\nu_1}} + \cdots + \chi_{P^{\nu_c}}. \]
By Lemma~\ref{lemma:FoulkesChars} we have
\[ \tag{$\dagger$}\psi^{\nu_1} + \cdots + \psi^{\nu_c} = \sum_{\mu \in \mathcal{E}_k(\gamma)}
\chi^\mu. \]
By hypothesis $\lambda$ is a maximal partition in the dominance order
on $\mathcal{E}_k(\gamma)$, 
and by~($\star$) each $\psi^{\nu_j}$ is a sum of ordinary irreducible
characters $\chi^\mu$ for partitions $\mu$ dominated by $\nu_j$. 
Therefore one of the partitions $\nu_j$ must equal $\lambda$, 
as required.
\end{proof}

We are now ready to prove Theorem~\ref{thm:main}

\begin{proof}[Proof of Theorem~\ref{thm:main}]
Suppose that the  projective summands of $H^{(2^m \; k)}$
lying in the block $B(\gamma, w_k(\gamma))$
are $P^{\nu_1}, \ldots, P^{\nu_c}$.
Then by ($\dagger$) above,
$\mathcal{E}_k(\gamma)$ has a partition into disjoint subsets 
$\mathcal{X}_1, \ldots, \mathcal{X}_c$ 
such that $\nu_j \in \mathcal{X}_j$ and
\[ \psi^{\nu_j} = \sum_{\mu \in \mathcal{X}_j} \chi^\mu \]
for each $j$.
It now follows from ($\star$)
that 
the column of the decomposition matrix of $S_n$ in characteristic~$p$ labelled
by $\nu_j$ has $1$s in the rows labelled by partitions in~$\mathcal{X}_j$, and
$0$s in all other rows.
\end{proof}

\section{Applications of Theorem~\ref{thm:main} and Proposition~\ref{prop:projectives}}

We begin with a precise statement of the result on diagonal Cartan numbers mentioned
in the introduction
after Proposition~\ref{prop:projectives}. 

\begin{theorem}[{\cite[Theorem~2.8]{Richards}} or {\cite[Proposition~4.6(i)]{BessenrodtUno}}]
\label{thm:bound}
If $\nu$ is a $p$-regular partition of~$n$ such that $\nu$
has weight $w$ then $d_{\mu \nu} \not= 0$ for at least $w+1$
distinct partitions $\mu$.
\end{theorem}

If
\hbox{$|\mathcal{E}_k(\gamma)| \le 2w_k(\gamma) + 1$} then it
follows from Theorems~\ref{thm:main} and~\ref{thm:bound} that
$\mathcal{E}_k(\gamma)$ has a unique maximal partition, say $\lambda$,
and the only non-zero entries of the column of the decomposition matrix
of $S_n$ labelled by $\lambda$ 
are $1$s in rows labelled by partitions in $\mathcal{E}_k(\gamma)$.
In these cases Theorem~\ref{thm:main} becomes a sharp result.

\begin{example}\label{ex:ex}
Firstly let $p=3$ and let $\gamma = (3,1,1)$.
We leave it to the reader to check that $w_0(\gamma) = 3$ and 
\[ \mathcal{E}_0(\gamma) = \{(8,4,2), (6,6,2), (6,4,4),
(6,4,2,2)\}.\] 
Hence the column of the decomposition matrix of $S_{12}$ in characteristic
$3$ labelled by $(8,4,2)$ has $1$s in the rows labelled by the four
partitions in $\mathcal{E}_0(\gamma)$ and no other non-zero entries. (The 
full decomposition matrix of the block $B( (3,1,1),3)$ was found by Fayers
in \cite[A.8]{Fayers13}.)

Secondly let $p=7$ and let $\gamma = (4,4,4)$. Then $w_6(\gamma) = 2$ and
$\mathcal{E}_6(\gamma) = \mathcal{X} \cup \mathcal{X}'$ where
\begin{align*}
\mathcal{X}  &= \{ (11,4,4,3,1^4), (11,4,4,2,1^5), 
                   (10,5,4,3,1^4), (10,5,4,2,1^5) \}, \\
\mathcal{X}' &= \{  (9,5,5,5,1,1), (9,5,5,4,1,1,1), (8,5,5,5,1,1,1) \}.
\end{align*}
The partitions in $\mathcal{X}$ and
$\mathcal{X}'$ are mutually incomparable under the dominance order.
Thus Theorem~\ref{thm:main}
determines the columns of the decomposition matrix of $S_{26}$ in characteristic $7$
labelled by $(11,4,4,3,1^4)$ and $(9,5,5,5,1,1)$.

Finally let $p=5$ and let $\gamma = (5,4,2,1^4)$. Then $w_6(\gamma) =
3$, and 
%\[ \mathcal{E}_6(\gamma) = \left\{ 
%(15,9,2,1^4),(15,6,5,1^4), (13,11,2,1^4), (13,6,5,3,1^3),(10,9,7,1^4),(10,9,5,3,1^3) \right\}. 
%\]
\[ \mathcal{E}_6(\gamma) = \left\{ 
\ttfrac{(15,9,2,1^4),(15,6,5,1^4), 
(13,11,2,1^4)}{(13,6,5,3,1^3),(10,9,7,1^4),(10,9,5,3,1^3)} \right\}. 
\]
It is easily seen that $w_1(\gamma) > 2$. (In fact $w_1(\gamma) = 8$.)
Therefore Theorem~\ref{thm:main} determines the column of the decomposition
matrix of $S_{30}$ in characteristic~$5$ labelled by $(15,9,2,1^4)$.
\end{example}

We now use the following
combinatorial lemma to 
prove that the bound in Theorem~\ref{thm:bound} is
attained in blocks of every weight.
Note that
when $p=3$ and $e=2$ the core used is $(3,1,1)$, as in the first example above.
For an introduction to James' abacus see \cite[page 78]{JK}.

\renewcommand{\d}{e}

\begin{lemma}\label{lemma:arbweight}
Let $p$ be an odd number, let $\d \in \N_0$, and let $\gamma$ be 
the $p$-core represented by the $p$-abacus with two beads on runner $1$, $e+1$
beads on runner~$p-1$, and one bead on every other runner. 
If $0 \le k \le \d+1$ then $w_k(\gamma) = \d+1-k$ and $|\mathcal{E}_k(\gamma)| = w_k(\gamma) + 1$.
\end{lemma}

\begin{proof}
The $p$-core $\gamma$ is represented by the abacus $A$. %shown in Figure~1 overleaf.
Moving the lowest $\d+1-k$ beads on runner $p-1$  down one step leaves
a partition %$\nu$ 
with exactly $k$ odd parts.
Therefore $w_k(\gamma) \le \d+1-k$.

Suppose that $\lambda$ is a partition with exactly~$k$
odd parts that can be obtained by 
a sequence of single step bead moves on $A$ in which exactly $\d-r$ beads
are moved on runner~$p-1$ and at most $\d+1-k$ moves are made in total.
 We may suppose that $r \ge k$ and that the beads on runner $p-1$ are moved
first, leaving an abacus $A^\star$.
Numbering rows %as in Figure~1,
so that row $0$ is the highest row,
let row $\ell$ be the lowest row of $A^\star$ to which any bead is moved in the subsequent moves.
Let $B$ be the abacus representing $\lambda$ 
that is obtained from $A^\star$ by making these moves.
The number of spaces before each beads on runner
$p-1$ in rows $\ell$, $\ell+1$, \ldots, $r$ 
is the same in both $A^\star$ and $B$, and is clearly odd in $A^\star$. Hence the
parts corresponding to these beads are odd. Therefore $\ell \ge r-k+1$.

If $B$ has a bead in row $\ell$ on a runner other than
runner $1$ or runner $p-1$, then this bead has been moved down from row $0$,
and so has been moved at least $\ell$ times.
The total number of moves made is at least $(\d-r) + \ell \ge \d-k+1$, and so
 $\ell = r-k+1$. But now $B$ has beads
corresponding to odd parts of~$\lambda$ on runner $p-1$ in row $0$,
as well as rows $\ell$, $\ell+1$,
\ldots, $r$, giving $k+1$ odd parts in total, a contradiction.

It follows that
the sequence of bead moves leading to $B$ may be reordered so that the
first $\d-r$ moves are made on runner $p-1$, and then the lowest
bead on runner $1$ is pushed down~$r-k$ times to row $r-k+1$.
The partition after these moves has $k+1$ odd parts. 
Moving the bead on runner $1$ down one step from row $r-k+1$ reduces the number
of odd parts by one, and is the only such move that does not move
a bead on runner $p-1$.
 Therefore $\mathcal{E}_k(\gamma)$ 
contains the partition  constructed at the start of the proof, and
one further partition 
for each $r \in \{0,1,\ldots, \d-k\}$. 
\end{proof}

%\begin{figure}[t]
%\begin{center}
%\hskip0.5in\scalebox{1}{\includegraphics{Figure2Abacus.pdf}}
%\caption{\small Abacus $A$ representing the $p$-core $\gamma$ in 
%Lemma~\ref{lemma:arbweight}.}
%\end{center}
%\end{figure}

Given an arbitrary weight $w\in\N$ and $k\in \N_0$, Lemma \ref{lemma:arbweight} 
gives an explicit 
partition $\lambda$ satisfying the hypothesis of Theorem \ref{thm:main} 
and such that $w_k(\gamma)=w$.
We use this in the following proposition.

\begin{proposition}\label{prop:arbweight}
Let $p$ be an odd prime and let $k$, $w \in \N_0$ be given. There
exists a $p$-core~$\gamma$ and a
partition $\lambda$ with $p$-core $\gamma$ and weight $w$ 
such that $\lambda$ has exactly $k$ odd parts and the
only  non-zero entries
in the column of the 
decomposition matrix labelled by $\lambda$ 
are $1$s lying in the $w+1$ rows labelled by elements of $\mathcal{E}_k(\gamma)$.
\end{proposition}

\begin{proof}
If $w=0$ and $k=0$ then take $\lambda = (2)$. Otherwise
let $\gamma$ be the $p$-core in Lemma~\ref{lemma:arbweight} when $\d = w+k-1$.
By this lemma we have  $w_k(\gamma) = w$. Moreover, if $k \ge p$ then $w_{k-p}(\gamma) = w + p$.
Taking $\lambda$
to be a maximal element of $\mathcal{E}_k(\gamma)$,
the proposition follows from Theorem~\ref{thm:main}
and Theorem~\ref{thm:bound}.
\end{proof}

We now turn to an application of Proposition~\ref{prop:projectives}.
Write $H^{(2^m \; k)}_R$ for the twisted Foulkes module defined over
a commutative ring $R$. 
Since the ordinary character of $H^{(2^m \; k)}_\Q$ is multiplicity-free,
the endomorphism algebra of $H^{(2^m \; k)}_F$ is commutative whenever
the field~$F$ has characteristic zero.
Hence the endomorphism
ring of $H^{(2^m \; k)}_\Z$ is commutative. This ring has a canonical $\Z$-basis
indexed by the double cosets of the  subgroup $S_2 \wr S_m \times S_k$
in $S_{2m+k}$. This basis makes it clear that the
canonical map 
\[ \End_{\Z S_{2m+k}}(H_\Z^{(2^m \; k)}) \rightarrow \End_{\F_pS_{2m+k}}(H_{\F_p}^{(2^m \; k)}) \]
is surjective, and so $\End_{FS_{2m+k}}(H_F^{(2^m \; k)})$ is commutative
for \emph{any} field $F$. 
This fact has some strong consequences for the structure of 
twisted Foulkes modules.

\begin{proposition}\label{prop:homs} 
Let $U$ and $V$ be distinct summands in a decomposition of $H^{(2^m\; k)}$,
defined over a field $F$,
into direct summands. Then $\End_{FS_{2m+k}}(U)$ is commutative and
$\Hom_{FS_{2m+k}}(U,V) = 0$.
\end{proposition}

\begin{proof}
Let $\pi_U$ be the projection map from $H^{(2^m \; k)}$
onto $U$ and let $\iota_U$ and $\iota_V$ be the inclusion maps of $U$ and $V$
respectively into $H^{(2^m \; k)}$.
Suppose that $\phi \in \Hom_{FS_{2n}}(U,V)$
is a non-zero homomorphism. Then $\pi_U \phi \iota_V$
does not commute with $\pi_U \iota_U$. (We compose homomorphisms from left to right.)
Moreover sending $\theta \in \End_{FS_{2m+k}}(U)$ to $\pi_U \theta \iota_U$
defines an injective map from $\End_{FS_{2m+k}}(U)$ into 
 the commutative algebra
$\End_{FS_{2m+k}}H^{(2^m \; k)}$.
\end{proof}

Proposition~\ref{prop:homs} implies that 
if $\lambda$ is a $p$-regular partition and~$P^\lambda$ is
a direct summand of $H^{(2^m \; k)}$, defined over a field
of characteristic~$p$, then there are no non-zero homomorphisms from
$P^\lambda$ to any other summand of $H^{(2^m \; k)}$.
Thus every composition factor of $H^{(2^m \; k)}$ isomorphic to $D^\lambda$
must come from  $P^\lambda$.
We also obtain the following corollary.

\begin{corollary}
Let $F$ be a field of odd characteristic. 
Given any $w \in \N$ there exists $n \in \N$ and an indecomposable projective
module $P^\lambda$ for $FS_n$ lying in a block of weight $w$ such that
$\End_{FS_n}(P^\lambda)$ is commutative.
\end{corollary}

\begin{proof}
Let $\gamma$ be the $p$-core in Lemma~\ref{lemma:arbweight} when $\d+1 = w$.
Taking $k=0$ we see that $w_0(\gamma) = w$.
If $\lambda$ is a maximal element of $\mathcal{E}_0(\gamma)$ 
then, by Proposition~\ref{prop:projectives}, $P^\lambda$ is a direct
summand of $H^{(2^m)}$, where $2m = |\lambda|$ and
both modules are defined over the field $F$.
The result now follows from Proposition~\ref{prop:homs}.
\end{proof}

\medskip

\section*{Acknowledgements}
The main theorem is a generalization of a result proved
by the second author in his D.~Phil thesis, under the supervision of  Karin Erdmann. He
gratefully acknowledges her support.

\def\cprime{$'$} \def\Dbar{\leavevmode\lower.6ex\hbox to 0pt{\hskip-.23ex
  \accent"16\hss}D}
\providecommand{\bysame}{\leavevmode\hbox to3em{\hrulefill}\thinspace}
\providecommand{\MR}{\relax\ifhmode\unskip\space\fi MR }
\providecommand{\MRhref}[2]{%
  \href{http://www.ams.org/mathscinet-getitem?mr=#1}{#2}
}
\providecommand{\href}[2]{#2}
\renewcommand{\MR}[1]{\relax}

\end{document}